\documentclass[11pt,english,british]{article}
\usepackage[T1]{fontenc}
\usepackage[latin9]{inputenc}
\usepackage{geometry}
\usepackage{mathrsfs}
\usepackage{amsmath}
\usepackage{amsthm}
\usepackage{amssymb}
\usepackage{graphicx}

\makeatletter

\providecommand{\tabularnewline}{\\}

\theoremstyle{plain}
\newtheorem{thm}{\protect\theoremname}[section]
\theoremstyle{plain}
\newtheorem{lem}[thm]{\protect\lemmaname}
\theoremstyle{remark}
\newtheorem{rem}[thm]{\protect\remarkname}
\theoremstyle{plain}
\newtheorem{prop}[thm]{\protect\propositionname}
\theoremstyle{definition}
\newtheorem{example}[thm]{\protect\examplename}

\usepackage[affil-it]{authblk}

\author[*]{Chunrong Feng} 
\author[*,**]{Huaizhong Zhao} 
\author[***]{Johnny Zhong}
\affil[*] {Department of Mathematical Sciences, Durham
University, DH1 3LE, UK}
\affil[**]{Research Centre for Mathematics and Interdisciplinary Sciences, Shandong University, Qingdao 266237, China}
\affil[***]{Department of Mathematical Sciences, Loughborough
University, LE11 3TU, UK}
\affil[ ]{chunrong.feng@durham.ac.uk, huaizhong.zhao@durham.ac.uk, J.Zhong@lboro.ac.uk}
\date{}

\newlength{\bibitemsep}
\newlength{\bibparskip}
\let\oldthebibliography\thebibliography
\renewcommand{\thebibliography}[1]{%
  \oldthebibliography{#1}%
  \setlength{\parskip}{\bibitemsep}%
  \setlength{\itemsep}{\bibparskip}%
}

\usepackage{babel}
\providecommand{\examplename}{Example}
\providecommand{\lemmaname}{Lemma}
\providecommand{\propositionname}{Proposition}
\providecommand{\remarkname}{Remark}
\providecommand{\theoremname}{Theorem}

\makeatother

\usepackage{babel}
\addto\captionsbritish{\renewcommand{\examplename}{Example}}
\addto\captionsbritish{\renewcommand{\lemmaname}{Lemma}}
\addto\captionsbritish{\renewcommand{\propositionname}{Proposition}}
\addto\captionsbritish{\renewcommand{\remarkname}{Remark}}
\addto\captionsbritish{\renewcommand{\theoremname}{Theorem}}
\addto\captionsenglish{\renewcommand{\examplename}{Example}}
\addto\captionsenglish{\renewcommand{\lemmaname}{Lemma}}
\addto\captionsenglish{\renewcommand{\propositionname}{Proposition}}
\addto\captionsenglish{\renewcommand{\remarkname}{Remark}}
\addto\captionsenglish{\renewcommand{\theoremname}{Theorem}}
\providecommand{\examplename}{Example}
\providecommand{\lemmaname}{Lemma}
\providecommand{\propositionname}{Proposition}
\providecommand{\remarkname}{Remark}
\providecommand{\theoremname}{Theorem}

\begin{document}
\title{Expected Exit Time for Time-Periodic Stochastic Differential Equations
and Applications to Stochastic Resonance}

\maketitle
\renewcommand{\labelenumi}{(\roman{enumi})}

\numberwithin{equation}{section}
\begin{abstract}
In this paper, we derive a parabolic partial differential equation
for the expected exit time of non-autonomous time-periodic non-degenerate
stochastic differential equations. This establishes a Feynman-Kac
duality between expected exit time of time-periodic stochastic differential
equations and time-periodic solutions of parabolic partial differential
equations. Casting the time-periodic solution of the parabolic partial
differential equation as a fixed point problem and a convex optimisation
problem, we give sufficient conditions in which the partial differential
equation is well-posed in a weak and classical sense. With no known
closed formulae for the expected exit time, we show our method can
be readily implemented by standard numerical schemes. With relatively
weak conditions (e.g. locally Lipschitz coefficients), the method
in this paper is applicable to wide range of physical systems including
weakly dissipative systems. Particular applications towards stochastic
resonance will be discussed.

\end{abstract}

\section{Introduction}

In many disciplines of sciences, (expected) exit time of stochastic
processes from domains is an important quantity to model the (expected)
time for certain events to occur. For example, time for chemical reactions
to occur \cite{Kramers,Gardiner_StochMethods,Zwanzig}, biological
neurons to fire \cite{Neurons_By_OU,Neurons_By_OU2}, companies to
default \cite{BlackCoxFPM,CreditRisk_FPM}, ions crossing cell membranes
in molecular biology \cite{FirstPassage_Biology} are all broad applications
of exit times. For autonomous stochastic differential equations (SDEs),
the expected exit time from a domain has been well-studied in existing
literature. In particular, it is well-known that the expected exit
time satisfies a second-order linear elliptic partial differential
equation (PDE) \cite{Hasminskii,Gardiner_StochMethods,Zwanzig,Pavliotis_LangevinTextbook,RiskenFP}.
However, in existing literature, it appears that the expected exit
time PDE is absent for non-autonomous SDEs and in particular time-periodic
SDEs. Our novel contribution is the rigorous derivation of a second-order
linear parabolic PDE obeyed by the expected exit time of time-periodic
SDEs as its time-periodic solution. This establishes a Feynman-Kac
duality for time-periodic SDEs for the expected duration. We expect
that our approach and this duality go beyond this current paper to
derive similar parabolic PDEs for other quantities associated to time-periodic
SDEs. Conversely, we expect this duality provides stochastic insight
into existing time-periodic solutions of parabolic PDEs. In this paper,
we also discuss briefly the ill-posedness of the PDE for the general
non-autonomous SDE case and thereby explaining its absence in literature.

With typical and relatively weak SDE conditions e.g. non-degenerate
diffusion and existence of continuous Markov transition density, the
PDE can be rigorously derived. In the interest of many physical systems,
we show that the results readily apply to weakly dissipative SDEs.
The conditions required to solve the PDE are weaker than that to derive
the PDE from the SDE. This is expected because from a PDE perspective,
weak solutions of PDE on bounded domains can often be attained requiring
coefficients to only be $L^{p}$ or Hölder; and classical solutions
may be obtained via Sobolev embedding. On the other hand, as a priori,
it is not known if the process would exit the bounded domain in finite
time or indeed have finite expectation. By considering the SDE and
its Markov transition probability on the entire unbounded domain,
we show that if the exit time has finite second moment then the PDE
derivation can be rigorously justified. We show that irreducibility
and the strong Feller property of the Markov transition probability
are the key ingredients to conclude the exit time has finite second
moment. While it is well-known that the strong Feller property holds
provided the coefficients are globally Hölder and bounded with uniformly
elliptic diffusion \cite{Friedman_ParabolicPDE,Stroock_Varadhan_MDP},
these conditions are too restrictive for the applications from a SDE
perspective. The celebrated Hörmander's condition is a weak condition
to deduce the strong Feller property for autonomous SDE \cite{HormanderBooks,Malliavin_Hypo,Hairer_Hormander}.
In the recent paper \cite{HopfnerLocherbachThieullenTimeDepLocalHormanders},
the authors extended Hörmander's condition to sufficiently imply the
strong Feller property holds for non-autonomous SDEs. The smoothness
SDE conditions of this paper is to invoke the result of \cite{HopfnerLocherbachThieullenTimeDepLocalHormanders}
while flexible enough for applications.

We provide two complementary approaches to prove that the parabolic
PDE has a unique solution in a weak and classical sense. In the proofs,
we keep as much generality as convenient to show the main ingredients
for the well-posedness of the PDE and for straightforward application
to similar problems. In one approach, we show that the time-periodic
solution can be casted as a fixed point of the parabolic PDE evolution
operator after a period. We prove that if the associated bilinear
form is coercive, then the time-periodic solution exists and is unique
by Banach Fixed Point Theorem. As coercivity can be difficult to verify
in practice, we also take a calculus of variation approach. Specifically,
we cast the problem as a convex optimisation problem by defining a
natural cost functional and show that a unique minimiser exists and
satisfies the PDE.

We emphasise that while our core results are theoretical in nature,
the Banach fixed point and convex optimisation approach can be readily
implemented by standard numerical schemes. Acquiring the tools to
numerically compute the expected exit time is vital because explicit
or even approximate closed form formulas for the expected time are
rarely known, even in the autonomous case. The known cases include
(autonomous) one-dimensional gradient SDEs with additive noise, where
the expected exit time can be expressed as a double integral \cite{Gardiner_StochMethods}
and has an approximate closed form solution given by Kramers' time,
when the noise is small \cite{Kramers}. Kramers' time has since been
extended to higher dimensional gradient SDEs \cite{Ber11}. However,
to our knowledge, there are currently no-known exact formulae for
the time-periodic case. Therefore, particularly for applications,
there is an imperative to solving the PDE numerically.

This paper together with the periodic measure concept provides a novel
mathematical approach to stochastic resonance, a phenomena that we
now briefly describe. In a series of papers \cite{BenziStochRes81,BenziStochRes,BenziStochRes83,NicolisPeriodicForcing},
the paradigm of stochastic resonance was introduced to explain Earth's
cyclical ice ages. In particular, the authors proposed a double-well
potential SDE with periodic forcing to model the scientific observation
that Earth's ice age transitions from ``cold'' and ``warm'' climate
occurs abruptly and almost regularly every $10^{5}$ years. The wells
model the two metastable states, where the process typically stays
at for large amount of time. The periodic forcing corresponds to the
annual mean variation in insolation due to changes in ellipticity
of the earth's orbit, while noise stimulates the global effect of
relatively short-term fluctuations in the atmospheric and oceanic
circulations on the long-term temperature behaviour. In the absence
of noise (with or without periodic forcing), these models do not produce
transition between the two metastable states. Similarly, in the absence
of periodic forcing, while the noise induces transitions between the
stable states, the transitions are not periodic. It is the delicate
interplay between periodicity and noise that explains the transitions
between the metastable states to be periodic. Since the seminal papers,
stochastic resonance has found applications in many physical systems
including optics, electronics, neuronal systems, quantum systems amongst
other applications \cite{GHJM98,JungHanggi,ZhouMossJung,JungPerioidcallyDrivenSys,HerrmannImkellerPavly,Lon93}.

The concept of periodic measures and ergodicity introduced in \cite{CFHZ2016}
provides a rigorous framework and new insight for understanding such
physical phenomena. Indeed, in \cite{FZZ19}, broad classes of SDEs
were shown to possess a unique geometric periodic measure and specifically
shown to apply double-well potential SDEs. The uniqueness of geometric
periodic measure implies transition between the wells\textbf{ }occurs
\cite{FZZ19}. While there is no standard definition \cite{JungHanggi,HerrmannImkeller},
stochastic resonance is said to occur if the expected transition time
between the metastable states is (roughly) half the period \cite{CherubiniStochRes}.
Indeed, the transition time between the wells is a special case of
exit time. Applying the theory developed in this paper, we show that
computationally solving the PDE and stochastic simulation for the
expected transition time agrees. We then fine tune the noise intensity
until the system exhibit stochastic resonance. Our PDE results also
show the transition between cold and ward climates is indeed very
abrupt when regime change happens.

Existing stochastic resonance literature often utilise Kramers' time,
note however that Kramers' time applies only to autonomous gradient
SDE case and in the small noise limit. For example in \cite{MW89,SR2state}
reduced the dynamics to \textquotedbl effective dynamics\textquotedbl{}
two-state time-homogeneous Markov process and invoked a time-perturbed
Kramers' time. More generally, utilising large deviation and specifically
Wentzell\textendash Freidlin theory \cite{FW98}, stochastic resonance
and related estimates can be attained in the small noise limit. For
example, \cite{MS01} attained estimates for escape rates, a closely
related quantity to expected transition time. Similarly, in \cite{IP01}
and \cite{HerrmannImkeller,HerrmannImkellerPavly,HerrmannSmallNoise},
the authors obtained estimates for the noise intensity for stochastic
resonance by reducing to two-state Markov process and time-independent
bounds respectively. In this paper, we retain the time-dependence
of the coefficients and furthermore, small and large noise are permissible.
In fact, the noise can even be state-dependent and exact exit time
duration is obtained.

\section{Expected Exit Time and Duration}

\subsection{Introduction}

Consider a stochastic process $(X_{t})_{t\geq s}$ on $\mathbb{R}^{d}$
with continuous sample-paths and an open non-empty (possibly unbounded)
domain $D\subset\mathbb{R}^{d}$ with boundary $\partial D$. Without
loss of generality, we assume throughout this paper that $D$ is connected.
Indeed if $D$ is disconnected, one can solve separately on each connected
subset. We define the first exit time from the domain $D$ (or first
passage time or first hitting time to the boundary) by 
\begin{align}
\eta_{D}(s,x) & :=\inf_{t\geq s}\{X_{t}\notin D|X_{s}=x\}=\inf_{t\geq s}\{X_{t}\in\partial D|X_{s}=x\},\label{eq:FirstExitTime}
\end{align}
where $x\in D$ and the equality holds by sample-path continuity.
We let $\eta_{D}(s,x)=\infty$ if $X_{t}$ never exits $D$. While
the absolute time in (\ref{eq:FirstExitTime}) is important, it is
mathematically convenient and practically useful to study instead
the exit duration 
\begin{align}
\tau_{D}(s,x) & :=\eta_{D}(s,x)-s\label{eq:Tau_Eta}
\end{align}
directly. As $D$ is generally fixed, where unambiguous, we omit the
subscript $D$ i.e. $\eta(s,x)=\eta_{D}(s,x)$ and $\tau(s,x)=\tau_{D}(s,x)$.
By Début theorem, $\eta(s,x)$ is both a hitting time and a stopping
time. In general, $\tau(s,x)$ is not. Thus some proofs and computations
will be first done for $\eta(s,x)$, then related to $\tau(s,x)$
via (\ref{eq:Tau_Eta}). In this paper, we are interested in their
expectations 
\begin{equation}
\bar{\eta}(s,x):=\mathbb{E}[\eta(s,x)],\quad\bar{\tau}(s,x):=\mathbb{E}[\tau(s,x)].\label{eq:TauExpected}
\end{equation}
In conventional notation, one typically writes $\bar{\eta}=\mathbb{E}^{s,x}[\eta]$
and $\bar{\tau}=\mathbb{E}^{s,x}[\tau]$. For subsequent proofs, it
is often more convenient that we keep the explicit dependence on the
random variables.

In this paper, we are specifically interested in the expected exit
and duration time for $T$-periodic non-degenerate SDEs on $\mathbb{R}^{d}$
of the form 
\begin{equation}
\begin{cases}
dX_{t}=b(t,X_{t})dt+\sigma(t,X_{t})dW_{t}, & t\geq s,\\
X_{s}=x, & x\in D,
\end{cases}\label{eq:GeneralSDE_NonAuton}
\end{equation}
where $W_{t}$ is a $d$-dimensional Brownian motion on a probability
space $(\Omega,\mathcal{F},\mathbb{P})$, $b\in C(\mathbb{R}\times\mathbb{R}^{d},\mathbb{R}^{d})$
and $\sigma\in C(\mathbb{R}\times\mathbb{R}^{d},\mathbb{R}^{d\times d})$
are $T$-periodic i.e. 
\[
b(t,\cdot)=b(t+T,\cdot),\quad\sigma(t,\cdot)=\sigma(t+T,\cdot),
\]
such that a unique solution $X_{t}=X_{t}^{s,x}$ exist. To avoid triviality,
we always assume the coefficients collectively have a minimal period
i.e. at least one of the coefficients have a minimal period.

When a unique solution of (\ref{eq:GeneralSDE_NonAuton}) exists,
one can define the Markovian transition probability 
\begin{equation}
P(s,t,x,\Gamma):=\mathbb{P}(X_{t}\in\Gamma|X_{s}=x),\quad s\leq t,\Gamma\in\mathcal{B}(\mathbb{R}^{d}).\label{eq:MarkovTransitionProb}
\end{equation}
If SDE (\ref{eq:GeneralSDE_NonAuton}) is $T$-periodic, then it is
straightforward to show that 
\begin{equation}
P(s,t,x,\cdot)=P(s+T,t+T,x,\cdot),\quad s\leq t.\label{eq:MarkovTransition_ModT}
\end{equation}

We refer to SDEs as non-autonomous when there is an explicit time-dependence,
periodic or otherwise. When the SDE coefficients are time-independent
i.e. $b(t,\cdot)=b(\cdot)$ and $\sigma(t,\cdot)=\sigma(\cdot)$,
then the SDE (\ref{eq:GeneralSDE_NonAuton}) is said to be autonomous.
It is well-known that for autonomous SDEs, the expected exit time
and expected duration coincide \cite{Gardiner_StochMethods,Pavliotis_LangevinTextbook,Zwanzig}.
Denoting both the expected exit and duration time by $\bar{\tau}(x)$,
it is moreover known that $\bar{\tau}(x)$ satisfies the following
second-order elliptic PDE with vanishing boundaries \cite{Hasminskii,Gardiner_StochMethods,Pavliotis_LangevinTextbook,Zwanzig,RiskenFP}
\begin{equation}
\begin{cases}
L\bar{\tau}=-1, & \text{in }D,\\
\bar{\tau}=0 & \text{on }\partial D,
\end{cases}\label{eq:ExpectedHittingPDEAuton}
\end{equation}
where 
\begin{equation}
Lf=\sum_{i=1}^{d}b^{i}(x)\partial_{i}f(x)+\frac{1}{2}\sum_{i,j=1}^{d}a^{ij}(x)\partial_{ij}^{2}f(x),\quad f\in C_{0}^{2}(\mathbb{R}^{d}),\label{eq:AutonStochGenerator}
\end{equation}
is the usual infinitesimal stochastic generator with the conventional
notation $\partial_{i}=\partial_{x_{i}}$ and $a(x)=(\sigma\sigma^{T})(x)$.

For non-autonomous SDEs however, due to the explicit dependence on
time, expected exit time and expected duration no longer coincide.
That is, $\bar{\tau}(s,x)$ generally depends on both initial time
and initial state. In this non-autonomous case, we write explicitly
the time-dependence and define the stochastic infinitesimal generator
of (\ref{eq:GeneralSDE_NonAuton}) by 
\begin{equation}
L(s)f(x)=\sum_{i=1}^{d}b^{i}(s,x)\partial_{i}f(x)+\frac{1}{2}\sum_{i,j=1}^{d}a^{ij}(s,x)\partial_{ij}^{2}f(x),\quad f\in C_{0}^{2}(\mathbb{R}^{d}),\label{eq:SpatialGenerator_fixed_s}
\end{equation}
and its adjoint (on $C_{0}^{2}(\mathbb{R}^{d})$), the Fokker-Planck
operator by 
\begin{equation}
L^{*}(s)f(x)=-\sum_{i=1}^{d}\partial_{i}(b^{i}(s,x)f(x))+\frac{1}{2}\sum_{i,j=1}^{d}\partial_{ij}\left(a^{ij}(s,x)f(x)\right),\quad f\in C_{0}^{2}(\mathbb{R}^{d}).\label{eq:FokkerPlanckGeneral}
\end{equation}

It is easy but important to see that for non-autonomous SDEs, $\bar{\tau}(s,x)$
does not satisfy ($\ref{eq:ExpectedHittingPDEAuton}$) even if $L$
is replaced by $L(s)$. We note also one approach is to consider lifted
coordinates $(t,X_{t}^{s,x})$. This autonomisation approach was noticed
by two of the authors in their work \cite{CFHZ2016}, where the terminology
\textquotedblleft lift\textquotedblright{} was used. In fact, the
lift leads to Markovian RDS (Random Dynamical System) cocycle and
homogeneous Markov semigroup. This allows us to obtain an invariant
measure from the periodic measure. The operator $\partial_{s}+L(s)$
appeared naturally as the infinitesimal generator of the lifted semigroup
and its spectral structure was also analysed to have an infinite number
of equally placed pure imaginary eigenvalues. This very interesting
phenomena is due to the degeneracy of $\partial_{s}+L(s)$ and periodic
boundary condition. This agrees with their spectral structure of ergodic
periodic measures of homogeneous Markov semigroup studied in the first
part of the paper which was already published in \cite{CFHZ2020}.
The lifting was heavily used in our earlier paper \cite{FZZ19}, where
we discussed Fokker-Planck equation for the density of periodic measure.
However for the problem we consider in this paper, we need to come
back to the spatial state space due to technical challenges that arises
when applying the techniques of autonomous system directly to the
lifted coordinates $(t,X_{t})$. We briefly survey these issues below.

We note that while the Markov transition probability of lifted coordinates
$(t,X_{t})$ is time-homogeneous, it immediately loses other important
properties such as irreducibility and strong Feller property. This
means classical results are not applicable. For instance, it is easy
to show the lifted process can never be ergodic. Some papers such
as \cite{BDE01} observed that $(t,X_{t})$ is time-homogeneous and
applied (\ref{eq:ExpectedHittingPDEAuton}) with operator $\partial_{s}+L(s)$.
However, upon close inspection of the derivation of (\ref{eq:ExpectedHittingPDEAuton})
from texts including \cite{Hasminskii,Gardiner_StochMethods,Pavliotis_LangevinTextbook,Zwanzig,RiskenFP},
we note that (\ref{eq:ExpectedHittingPDEAuton}) was derived only
for autonomous systems, rather than non-autonomous systems.

Directly applying the autonomous results for non-autonomous SDEs with
lifted coordinates $(t,X_{t})$ leads to technical boundary conditions
issues. For instance, from a SDE perspective, while (\ref{eq:FirstExitTime})
is well-defined in spatial coordinates, it is not immediately clear
how to define exit time in $(t,X_{t})$ coordinates. Possibly the
most natural choice is $\pi_{x}(t,X_{t})$, where $\pi_{x}$ is the
spatial projection. This adds unnecessary complexity and leads to
further issues from a PDE perspective. Specifically, suppose we can
extend the classic result from \cite{Hasminskii,Gardiner_StochMethods,Pavliotis_LangevinTextbook,Zwanzig,RiskenFP}
directly to $(t,X_{t})$, then it suggests we replace the boundary
conditions of ($\ref{eq:ExpectedHittingPDEAuton}$) with $u(s,x)=0$
on $\partial(\mathbb{R^{+}}\times D)=\left(\{0\}\times D\right)\cup\left(\mathbb{R}^{+}\times\partial D\right)$.
This is problematic. Treated as a parabolic PDE, this suggest that
if the process starts at time $s=0$, then for any $x\in D$, the
exit time to leave the domain is zero. This cannot be the case. If
persisted to interpreted as an elliptic equation, the elliptic PDE
would be degenerate and the boundary conditions is insufficient (as
conditions on $\{T\}\times D$ is missing, assuming we cap $\mathbb{R}^{+}$
to some finite $T$). If we diverge from \cite{Hasminskii,Gardiner_StochMethods,Pavliotis_LangevinTextbook,Zwanzig,RiskenFP}
slightly, the intuitive boundary condition is again $\pi_{x}u(s,x)=0$.
This is problematic in that there are minimal (if any) PDE theory
that deals with such projection boundary conditions therefore making
attempts to prove existence and uniqueness difficult. It is our assumption
that particular to the boundary conditions is why \cite{Hasminskii,Gardiner_StochMethods,Pavliotis_LangevinTextbook,Zwanzig,RiskenFP}
and other similar texts have omitted the non-autonomous SDE case.

The novel contribution of this paper is the rigorous derivation of
the second-order parabolic PDE (\ref{eq:ExpectedDurationPDE_Periodic})
in which $\bar{\tau}$ satisfies for $T$-periodic SDEs. This is complete
with boundary conditions in space and periodicity in the time domain
to show the PDE is well-posed. Furthermore, using the framework we
build in this paper, we provide numerical methods for solving the
periodic solution of the parabolic PDE.

\subsection{Expected Duration PDE Derivation}

To rigorously derive the expected duration PDE, we first fix some
standard nomenclature and notation. For the open domain $D\subset\mathbb{R}^{d}$
and open interval $I\subset\mathbb{R}^{+}$. We define their Cartesian
product by $D_{I}:=I\times D$. When $I=(0,T)$, we define $D_{T}:=(0,T)\times D$.
And we define $B_{r}(y):=\{x\in\mathbb{\mathbb{R}}^{d}|\left\lVert x-y\right\rVert <r\}$
for the open ball of radius $r>0$ centred at $y$, and denote for
convenience $B_{r}:=B_{r}(0)$. On $\mathbb{R}^{d}$, we let $\Lambda$
be the Lebesgue measure. For matrices, we let $L_{2}(\mathbb{R}^{d}):=\{\sigma\in\mathbb{R}^{d\times d}|\lVert\sigma\rVert_{2}<\infty\}$,
where $\lVert\sigma\rVert_{2}=\sqrt{\text{Tr}(\sigma\sigma^{T})}=\sqrt{\sum_{i,j=1}^{d}\sigma_{ij}^{2}}$
is the standard Frobenius norm. For $\theta_{x}\in(0,1]$, denote
by $C^{\theta_{x}}(D)$ the collection of all functions globally $\theta_{x}$-Hölder
continuous on $D$. For $\theta_{t},\theta_{x}\in(0,1]$, denote by
$C^{\theta_{t},\theta_{x}}(I\times D)$ the set of functions $\theta_{t}$-Hölder
and $\theta_{x}$-Hölder functions in the $t$ and $x$ variable respectively.

Let $k_{t},k_{x}\in\mathbb{N}$, we denote by $C^{k_{t},k_{x}}(I\times D)$
to be the space of continuously $k_{t}$-differentiable functions
in $t$ and continuously $k_{x}$-differentiable function in $x$.
For $\theta_{t},\theta_{x}\in(0,1]$, $C^{k_{t}+\theta_{t},k_{x}+\theta_{t}}(I\times D)$
denotes the space of $C^{k_{t},k_{x}}(I\times D)$ functions in which
the $k_{t}$'th $t$-derivative and $k_{x}$'th $x$-derivatives are
$\theta_{t}$ and $\theta_{x}$ are Hölder respectively. We also let
$C_{b}^{\infty}(B_{n})$ denote the space of bounded infinitely differentiable
real-valued functions on $B_{n}$. Define for ease, $\lVert\sigma\rVert_{\infty}:=\sup_{(t,x)\in\mathbb{R}^{+}\times\mathbb{R}^{d}}\lVert\sigma(t,x)\rVert_{2}$.
Following the conditions required of Theorem 1 of \cite{HopfnerLocherbachThieullenTimeDepLocalHormanders},
we say that drift is said to be locally smooth and bounded if for
all $n\in\mathbb{N},$ 
\begin{equation}
b(t,x)+\partial^{\beta}b(t,x)\quad\text{bounded on }\mathbb{R}^{+}\times B_{n},\label{eq:LocallySmoothDrift}
\end{equation}
where $\beta=(\beta_{0},\beta_{1},...,\beta_{d})\in\mathbb{N}^{d+1}$,
$\lvert\beta\rvert:=\sum_{i=0}^{d}\beta_{i}=d$ and $\partial^{\beta}:=\frac{\partial^{\lvert\beta\rvert}}{\partial_{t}^{\beta_{0}}\partial_{x_{1}}^{\beta_{1}}\cdot\cdot\cdot\partial_{x_{d}}^{\beta_{d}}}$.
Note that the partial derivative here refers to derivative in any
spatial direction and in the time direction.

We say the SDE (\ref{eq:GeneralSDE_NonAuton}) satisfies the regularity
condition if its coefficients $b$ and $\sigma$ are locally Lipschitz
and there exists a function $V\in C^{1,2}(\mathbb{R}^{+}\times\mathbb{R}^{d},\mathbb{R}^{+})$
and a constant $c>0$ such that $\lim_{x\rightarrow\infty}V(t,x)=\infty$
for all fixed $t$ and 
\begin{equation}
L(t)V\leq cV,\quad\text{on }\mathbb{R}^{+}\times\mathbb{R}^{d}.\label{eq:RegularityCondition}
\end{equation}
It was shown in \cite{Hasminskii} that if SDE (\ref{eq:GeneralSDE_NonAuton})
satisfies the regularity condition (\ref{eq:RegularityCondition}),
then the process is regular i.e. $\mathbb{P}^{s,x}\{\eta_{\infty}=\infty\}=1$,
where $\eta_{\infty}=\lim_{n\rightarrow\infty}\eta_{B_{n}}$. Moreover,
there exists a unique almost surely finite solution. SDE (\ref{eq:GeneralSDE_NonAuton})
is said to be weakly dissipative if there exists a constant $c\geq0,$
$\lambda>0$ such that 
\begin{equation}
2\langle b(t,x),x\rangle\leq c-\lambda\lVert x\rVert^{2}\quad\text{on \ensuremath{\mathbb{R}^{+}\times\mathbb{R}^{d}}}.\label{eq:WeaklyDissip}
\end{equation}
If $c=0$, then it is said to be dissipative. While weak dissipativity
is a stronger condition than (\ref{eq:RegularityCondition}) and is
also often easier to verify, particularly for many typical physical
systems. It was shown that $T$-periodicity and weak dissipativity
leads to the geometric ergodicity of periodic measures \cite{FZZ19}.

We say $\sigma$ is locally smooth and bounded if for all $n\in\mathbb{N}$
\begin{equation}
\sigma^{ij}\in C_{b}^{\infty}(B_{n}),\quad1\leq i,j\leq d.\label{eq:LocallySmoothSigma}
\end{equation}
Finally, we say $\sigma$ is (globally) bounded with (globally) bounded
inverse if 
\begin{equation}
\max\{\lVert\sigma\rVert_{\infty},\lVert\sigma^{-1}\rVert_{\infty}\}<\infty.\label{eq:Sigma_and_inverse_bounded}
\end{equation}
Observe (\ref{eq:LocallySmoothSigma}) and (\ref{eq:LocallySmoothDrift})
imply the respective functions are locally Lipschitz. Whenever we
assume (\ref{eq:LocallySmoothSigma}), we always demand that $\sigma$
is a function of spatial variables only.

It appears that in numerous existing literature, almost surely finite
exit time is implicitly assumed. Particularly for degenerate noise,
it may well be that the exit time is infinite with positive probability
or indeed almost surely. Utilising asymptotic stability of diffusion
processes, it is easy to construct examples where the process never
leaves a point or domain. We refer readers to \cite{Mao_SDEs} for
examples. In the following lemma, we give verifiable conditions to
imply irreducibility and show further that $\eta$ is almost surely
finite with finite first and second moments.
\begin{lem}
\label{lem:AlmostSurelyFinite_Eta_and_Expectation}Let $D\subset\mathbb{R}^{d}$
be a non-empty open bounded set. Assume that the $T$-periodic SDE
(\ref{eq:GeneralSDE_NonAuton}) satisfies (\ref{eq:RegularityCondition}),
(\ref{eq:LocallySmoothDrift}), (\ref{eq:LocallySmoothSigma}) and
(\ref{eq:Sigma_and_inverse_bounded}). Then $\eta(s,x)$ is finite
almost surely for all $(s,x)\in\mathbb{R}^{+}\times D$. Moreover,
$\eta(s,x)$ has finite first and second moments. 
\end{lem}

\begin{proof}
It was shown in \cite{FZZ19} that (\ref{eq:RegularityCondition})
and (\ref{eq:Sigma_and_inverse_bounded}) sufficiently implies $P$
is irreducible i.e. $P(s,t,x,\Gamma)>0$ for all $x\in\mathbb{R}^{d}$,
$0\leq s<t<\infty$ and non-empty open set $\Gamma\in\mathcal{B}(\mathbb{R}^{d})$.
Then for any fixed $s\in\mathbb{R}^{+}$, for all $x\in\mathbb{R}^{d}$,
it follows that there exists an $\epsilon(x)=\epsilon(s,x,D)\in(0,1)$
such that 
\[
\epsilon(x)=P(s,s+T,x,D).
\]
By the results of \cite{HopfnerLocherbachThieullenTimeDepLocalHormanders},
\cite{FZZ19} showed that when conditions (\ref{eq:RegularityCondition}),
(\ref{eq:LocallySmoothDrift}), (\ref{eq:LocallySmoothSigma}) and
(\ref{eq:Sigma_and_inverse_bounded}) hold then $P$ possesses a smooth
density. This implies that $P$ is strong Feller i.e. $P(s,t,\cdot,\Gamma)$
is continuous for all $s<t$ and $\Gamma\in\mathcal{B}(\mathbb{R}^{d})$.
Then it follows from the boundedness of $D$ that the probability
of staying within $D$ in one period is at most 
\[
\epsilon:=\sup_{x\in\bar{D}}\epsilon(x)>0.
\]
Since $P$ is irreducible and $D\subset\mathbb{R}^{d}$ is non-empty
open, we have that $P(s,s+T,x,D^{c})>0$ for any $x\in D$. Further,
since $D$ is bounded, we deduce that $\epsilon<1$. By (\ref{eq:MarkovTransition_ModT}),
$Z^{s,x}=(Z_{n}^{s,x}):=(X_{s+nT}^{s,x})_{n\in\mathbb{N}}$ is time-homogeneous
Markov chain with one-step Markovian transition $P(s,s+T,x,\cdot)$.
Define the exit time 
\[
\eta_{Z}=\eta_{Z}(s,x):=\min\{n\in\mathbb{N}:Z_{n}^{s,x}\notin D\}.
\]
By sample-path continuity of $X_{t}$, it is clear that $X_{t}\notin D$
for at least one $t\in[s+(\eta_{Z}-1)T,s+\eta_{Z}T]$. Hence $\tau(s,x):=\eta(s,x)-s\leq\eta_{Z}(s,x)\cdot T$,
in particular, we have 
\[
\{\eta_{Z}(s,x)\geq n\}\subset\{\eta(s,x)\geq s+nT\}.
\]
Hence if $\mathbb{P}(\eta_{Z}<\infty)=1$ then $\mathbb{P}(\eta<\infty)=1$
i.e. if $Z_{n}^{s}$ leaves $D$ in almost surely finite time then
$X_{t}$ does also. For any $n\in\mathbb{N}$, it is easy to see that
\[
\{\eta_{Z}=n\}=\{Z_{n}^{s}\in D^{c}\}\cap\bigcap_{m=1}^{n-1}\{Z_{m}^{s}\in D\}.
\]
Since $Z_{0}^{s}=x\in D$, by elementary time-homogeneous Markov chain
properties, 
\begin{align}
\mathbb{P}(\eta_{Z}=n) & =\mathbb{P}(Z_{n}^{s}\in D^{c}|\bigcap_{m=0}^{n-1}\{Z_{m}^{s}\in D\})\mathbb{P}(\bigcap_{m=0}^{n-1}\{Z_{m}^{s}\in D\})\nonumber \\
 & =\mathbb{P}(Z_{n}^{s}\in D^{c}|Z_{n-1}^{s}\in D)\prod_{m=1}^{n-1}\mathbb{P}(Z_{m}^{s}\in D|Z_{m-1}^{s}\in D)\nonumber \\
 & \leq\epsilon^{n-1}.\label{eq:Prob(eta=00003D00003Dn)}
\end{align}
This concludes that $\eta$ is almost surely finite. Via (\ref{eq:Prob(eta=00003D00003Dn)}),
it is elementary to show that $\tau$ has finite first and second
moments: 
\[
\mathbb{E}[\tau(s,x)]\le T\mathbb{E}[\eta_{Z}(s,x)]=T\sum_{n=0}^{\infty}n\mathbb{P}(\eta_{Z}=n)\leq T\sum_{n=0}^{\infty}\frac{d}{d\epsilon}\epsilon^{n}=T\frac{d}{d\epsilon}\frac{1}{1-\epsilon}=\frac{T}{\left(1-\epsilon\right)^{2}}<\infty.
\]
Similarly, 
\[
\mathbb{E}[\tau^{2}(s,x)]\leq T^{2}\mathbb{E}[\eta_{Z}^{2}(s,x)]=T^{2}\sum_{n=0}^{\infty}\left[\frac{d^{2}}{d\epsilon^{2}}\epsilon^{n+1}-\frac{d}{d\epsilon}\epsilon^{n}\right]=\frac{T^{2}(1+\epsilon)}{(1-\epsilon)^{3}}<\infty.
\]
It follows that $\eta$ has finite first and second moments.
\end{proof}
\begin{rem}
Observe that Lemma \ref{lem:AlmostSurelyFinite_Eta_and_Expectation}
abstractly holds provided that $P$ is irreducible and strong Feller.
It is well-known that for autonomous SDEs, Hörmander's condition sufficiently
implies the existence of a smooth density for $P$ and therefore implies
the strong Feller property. However, we note that Hörmander's condition
is not sufficient for Lemma \ref{lem:AlmostSurelyFinite_Eta_and_Expectation}
to hold. Firstly, Hörmander's condition is insufficient to imply irreducibility,
we refer readers to Remark 2.2 of \cite{Hairer_Hormander} for a counterexample.
Secondly, Hörmander's condition was classically written for autonomous
systems hence not directly applicable to the current non-autonomous
case. Inclusive of the non-autonomous, it is well-known that density
of $P$ exists for uniformly elliptic diffusions with globally Hölder
and bounded coefficients in $\mathbb{R}^{d}$ \cite{FW98,Stroock_Varadhan_MDP}.
However, these conditions can be too restrictive for applications
from the SDE perspective. On the other hand, Theorem 1 of \cite{HopfnerLocherbachThieullenTimeDepLocalHormanders}
extends Hörmander's condition to the case of non-autonomous SDEs with
the conditions (\ref{eq:LocallySmoothDrift}), (\ref{eq:RegularityCondition})
and (\ref{eq:LocallySmoothSigma}). While these conditions require
more smoothness than globally Hölder mentioned, the coefficients can
be unbounded and is flexible enough for a wide range of SDEs.
\end{rem}

\begin{rem}
It should be clear that Lemma \ref{lem:AlmostSurelyFinite_Eta_and_Expectation}
can be adapted to hold in the more general (not-necessarily $T$-periodic)
non-autonomous case. Namely by picking any fixed $T>0$, define $\epsilon_{n}:=\sup_{x\in\overline{D}}P(s+(n-1)T,s+nT,x,D)$,
then the same calculations via properties of the two-parameter Markov
kernel yields $\mathbb{P}(\eta_{Z}=n)\leq\epsilon^{n}$, where $\epsilon:=\max_{n\in\mathbb{N}}\epsilon_{n}$. 
\end{rem}

The Fokker-Planck equation is a well-known second-order linear parabolic
PDE that describes the time evolution of the probability density function
associated to SDEs \cite{BKRS15,RiskenFP,Hasminskii,Pavliotis_LangevinTextbook,Gardiner_StochMethods,Zwanzig}.
The existence and uniqueness of Fokker-Planck equation have been studied
in many settings including irregular coefficients and time-dependent
coefficients \cite{LL08,BKRS15,RZ10,DR12}. In \cite{FZZ19}, it was
shown that the periodic measure density is necessarily and sufficiently
the time-periodic solution of the Fokker-Planck equation. Existence
of time-periodic solution of the Fokker-Planck has been discussed
in \cite{JungNumericalSchemeForLiftedFokkerPlanck,FP1,FP2}.

To study the exit problem, we study the Fokker-Planck equation in
the domain $D$ and impose absorbing boundaries \cite{RiskenFP,Gardiner_StochMethods,Pavliotis_LangevinTextbook}.
Specifically, let $p_{D}(s,t,x,y)$ denote the probability density
of the process starting at $x$ at time $s$ to $y$ at time $t$
that gets absorbed on $\partial D$. Then the density $p_{D}$ satisfies
the following Fokker-Planck equation 
\begin{equation}
\begin{cases}
\partial_{t}p_{D}(s,t,x,y)=L^{*}(t)p_{D}(s,t,x,y),\\
p_{D}(s,s,x,y)=\delta_{x}(y), & x\in D,\\
p_{D}(s,t,x,y)=0, & \mbox{if }y\in\partial D,t\geq s.
\end{cases}\label{eq:FokkerPlanck_AbsorbingBoundaries}
\end{equation}
Here $L^{*}(t)$ acts on forward variable $y$. To discuss the solvability
of (\ref{eq:FokkerPlanck_AbsorbingBoundaries}) and subsequent PDEs,
we lay out typical PDE conditions that are weaker than conditions
required by Lemma \ref{lem:AlmostSurelyFinite_Eta_and_Expectation}.\\

\textbf{Condition A1}: For some $\theta\in(0,1]$,
\begin{enumerate}
\item Domain $D\in\mathcal{B}(\mathbb{R}^{d})$ is non-empty and open with
boundary $\partial D\in C^{\theta}(\mathbb{R}^{d-1})$. 
\item The coefficients $a^{ij},b^{i}\in C^{\frac{\theta}{2},\theta}(\bar{D}_{T})$. 
\item The matrix $a(s,x)=(a^{ij}(s,x))$ is uniformly elliptic i.e. there
exists $\alpha>0$ such that
\end{enumerate}
\begin{equation}
\langle a(s,x)\xi,\xi\rangle_{\mathbb{R}^{d}}\geq\alpha\left\lVert \xi\right\rVert _{\mathbb{R}^{d}}^{2},\quad(s,x)\in D_{T},\xi\mathbb{\in R}^{d}.\label{eq:UniformlyElliptic}
\end{equation}
Particularly for adjoint operator $L^{*}(t)$ where more differentiability
is required, we consider further\\

\textbf{Condition A2}: For some $\theta\in(0,1]$, Condition A1 holds
and moreover $a^{ij},b^{i}\in C^{1+\theta,2+\theta}(\bar{D}_{T})$
and $\partial D\in C^{2+\theta}(\mathbb{R}^{d-1})$.

It is well-known that if Condition A2 holds, then there exists a unique
solution $p_{D}(s,\cdot,x,\cdot)\in C^{1,2}(D_{T})$ to (\ref{eq:FokkerPlanck_AbsorbingBoundaries}).
Moreover, $p_{D}(s,t,x,y)$ is jointly continuous in $(x,y)$. For
details, we refer readers to Section 7, Chapter 3 in \cite{Friedman_ParabolicPDE}.
The following lemma and its proof are similar to the one presented
in \cite{Gardiner_StochMethods,Pavliotis_LangevinTextbook,RiskenFP}
when the coefficients are time-independent. We prove for the time-dependent
coefficients case. For clarity of the key ingredients of the following
lemma, we assume $\eta$ to have finite second moment rather than
the conditions assumed in Lemma \ref{lem:AlmostSurelyFinite_Eta_and_Expectation}.
\begin{lem}
\label{lem:TauExpression} Assume that Condition A2 holds for SDE
(\ref{eq:GeneralSDE_NonAuton}). Assume further that $\eta$ has finite
second moment. Then 
\begin{equation}
\bar{\tau}(s,x)=\int_{s}^{\infty}\int_{D}p_{D}(s,t,x,y)dydt,\label{eq:TauExpression}
\end{equation}
where $p_{D}(s,\cdot,x,\cdot)$ is the unique solution to (\ref{eq:FokkerPlanck_AbsorbingBoundaries}). 
\end{lem}

\begin{proof}
Let $G(s,t,x)$ be the probability that the process starting at $x$
at time $s$ is still within $D$ at time $t\geq s$. In the derivation
below, we treat $(s,x)$ as fixed parameters so that $G$ is only
a function of $t$. By the absorbing boundary conditions of $p_{D}$,
we have 
\begin{equation}
G(s,t,x)=\int_{D}p_{D}(s,t,x,y)dy.\label{eq:G_expressed_as_integral}
\end{equation}
On the other hand, 
\[
G(s,t,x)=\mathbb{P}(\eta(s,x)>t)=1-\mathbb{P}(\eta(s,x)\leq t).
\]
Then, since $p_{D}$ is $t$-differentiable, by (\ref{eq:G_expressed_as_integral}),
it is clear that a density $p_{\eta}(s,t,x)$ exists for $\eta(s,x)$
given by 
\begin{equation}
p_{\eta}(s,t,x)=-\partial_{t}G(s,t,x).\label{eq:G-density}
\end{equation}
Note that if $x\in D$ then $G(s,s,x)=1$. Note further that by Chebyshev's
inequality, 
\[
G(s,t,x)=\mathbb{P}(\eta(s,x)>t)\leq\frac{1}{t^{2}}\mathbb{E}\left[\eta^{2}(s,x)\right],\quad t>s.
\]
Since $G\geq0$, it follows that $\lim_{t\rightarrow\infty}tG(s,t,x)=0$,
hence the following holds by an integration by parts 
\begin{align*}
\bar{\eta}(s,x) & =\int_{s}^{\infty}tp_{\eta}(s,t,x)dt\\
 & =-\int_{s}^{\infty}t\partial_{t}G(s,t,x)dt\\
 & =-tG(s,t,x)|_{t=s}^{\infty}+\int_{s}^{\infty}G(s,t,x)dt\\
 & =s+\int_{s}^{\infty}G(s,t,x)dt.
\end{align*}
Hence 
\begin{equation}
\bar{\tau}(s,x)=\int_{s}^{\infty}G(s,t,x)dt.\label{eq:Tau_Equals_Integral_G}
\end{equation}
The result follows by (\ref{eq:G_expressed_as_integral}). 
\end{proof}
While finite first moment of $\eta$ was not explicitly used in Lemma
\ref{lem:TauExpression}, we note that it is of course finite since
it has finite second moment and applying Hölder's inequality. It is
then obvious then that (\ref{eq:Tau_Equals_Integral_G}) is finite.

Let $X^{0}$ and $X^{1}$ be two random variables, we write $X^{0}\sim X^{1}$
if they have the same distribution. Then we have the following intuitive
lemma that was proved and presented in \cite{FZZ19}. 
\begin{lem}
\label{lem:DistributionModT} Let Condition A2 hold for $T$-periodic
SDE (\ref{eq:GeneralSDE_NonAuton}). Let $\left(X_{t}^{0}\right)_{t\geq s},\left(X_{t}^{1}\right)_{t\geq s+T}$
be two processes satisfying (\ref{eq:GeneralSDE_NonAuton}). If $X_{s}^{0}\sim X_{s+T}^{1}$
then $X_{s+t}^{0}\sim X_{s+T+t}^{1}$ for all $t\geq0.$ 
\end{lem}

For $T$-periodic SDEs, we show in the next lemma, that the expected
duration $\bar{\tau}$ is also $T$-periodic. While this holds in
expectation, the same cannot be said of the sample-path realisations
of $\tau$. This is essentially because the noise realisation is not
periodic! In the context of random dynamical systems, this can be
proven rigorously. Indeed, if $\omega$ denotes the noise realisation
and $\theta_{t}$ to be the Wiener shift, then one has $\tau(s,x,\omega)=\tau(s+T,x,\theta_{T}\omega)$,
see \cite{CFHZ2016} for further details.
\begin{lem}
\label{lem:TauIsPeriodic}Assume that Condition A2 holds for $T$-periodic
SDE (\ref{eq:GeneralSDE_NonAuton}). Assume further that $\eta$ has
finite second moment. Then $\bar{\tau}$ is also $T$-periodic. 
\end{lem}

\begin{proof}
By Lemma \ref{lem:DistributionModT} and (\ref{eq:TauExpression}),
we have 
\begin{align*}
\bar{\tau}(s,x) & =\int_{s}^{\infty}\int_{D}p_{D}(s,r,x,y)dydr\\
 & =\int_{s}^{\infty}\int_{D}p_{D}(s+T,r+T,x,y)dydr\\
 & =\int_{s+T}^{\infty}\int_{D}p_{D}(s+T,r,x,y)dydr\\
 & =\bar{\tau}(s+T,x).
\end{align*}
\end{proof}
For the following theorem, we recall Kolmogorov's backward equation
\begin{equation}
\partial_{s}p(s,t,x,y)+L(s)p(s,t,x,y)=0,\label{eq:KolmBackEqn}
\end{equation}
where $L(s)$ acts on $x$ variable.

We are now ready to derive the PDE in which $\bar{\tau}(s,x)$ satisfies.
When the SDE is $T$-periodic, we show $\bar{\tau}(s,x)$ is the $T$-periodic
solution of a second-order linear parabolic PDE. This contrasts with
the autonomous case where the expected exit time satisfies the second-order
linear elliptic PDE (\ref{eq:ExpectedHittingPDEAuton}). To our knowledge
the derived PDE and particularly its interpretation is new in literature.
We note further that the following theorem establishes a Feynman-Kac
duality for time-periodic SDEs for the expected duration. 
\begin{thm}
\label{thm:ExpectedDurationPDE_Theorem}Assume $T$-periodic SDE (\ref{eq:GeneralSDE_NonAuton})
satisfies the same conditions as Lemma \ref{lem:AlmostSurelyFinite_Eta_and_Expectation}.
Then the expected duration $\bar{\tau}$ is the periodic solution
of the following partial differential equation of backward type

\begin{equation}
\begin{cases}
\partial_{s}u(s,x)+L(s)u(s,x)=-1, & \text{in }D_{T},\\
u=0, & \mbox{on }[0,T]\times\partial D,\\
u(0,\cdot)=u(T,\cdot). & \text{on }D.
\end{cases}\label{eq:ExpectedDurationPDE_Periodic}
\end{equation}
\end{thm}

\begin{proof}
By Lemma \ref{lem:AlmostSurelyFinite_Eta_and_Expectation}, $\eta$
has finite second moment and Condition A2 holds. Hence Lemma \ref{lem:TauExpression}
holds. Thus, by (\ref{eq:Tau_Equals_Integral_G}), observe that for
any $\delta>0$, 
\begin{align*}
\bar{\tau}(s+\delta,x)-\bar{\tau}(s,x) & =\int_{s+\delta}^{\infty}\left(G(s+\delta,t,x)-G(s,t,x)\right)dt-\mathcal{G}(s+\delta),
\end{align*}
where for clarity, $\mathcal{G}(r):=\int_{s}^{r}G(s,t,x)dt$. It follows
by the fundamental theorem of calculus that 
\begin{align*}
\partial_{s}\bar{\tau}(s,x) & =\int_{s}^{\infty}\partial_{s}G(s,t,x)dt-\mathcal{G}'(s)\\
 & =\int_{s}^{\infty}\int_{D}\partial_{s}p_{D}(s,t,x,y)dydt-G(s,s,x)\\
 & =\int_{s}^{\infty}\int_{D}\partial_{s}p_{D}(s,t,x,y)dydt-1,
\end{align*}
where recall that $G$ is expressed by (\ref{eq:G_expressed_as_integral})
and $G(s,s,x)=1$ since $x\in D$. Acting $L(s)$ on $\bar{\tau}$
by (\ref{eq:TauExpression}) and (\ref{eq:KolmBackEqn}), we have
\begin{align*}
L(s)\bar{\tau}(s,x) & =\int_{s}^{\infty}\int_{D}L(s)p_{D}(s,t,x,y)dydt=-\int_{s}^{\infty}\int_{D}\partial_{s}p_{D}(s,t,x,y)dydt.
\end{align*}
Summing these quantities yields 
\begin{equation}
(\partial_{s}+L(s))\bar{\tau}(s,x)=-1.\label{eq:ExpecedDurationPDE_General}
\end{equation}
For $T$-periodic systems, Lemma \ref{lem:TauIsPeriodic} showed that
$\bar{\tau}(s,\cdot)=\bar{\tau}(s+T,\cdot)$ for all $s\in\mathbb{R}^{+}$
hence deducing $\bar{\tau}$ satisfies (\ref{eq:ExpectedDurationPDE_Periodic})
and $u$ is $T$-periodic. By Lemma \ref{lem:TauIsPeriodic}, this
is sufficient by imposing $u(0,\cdot)=u(T,\cdot)$ and the result
follows.
\end{proof}
\begin{rem}
\label{rem:Insolvability_of_general_nonauton} In the proof of Theorem
\ref{thm:ExpectedDurationPDE_Theorem}, note that $T$-periodicity
was not assumed until (\ref{eq:ExpecedDurationPDE_General}). This
suggests that for general non-autonomous (not necessarily periodic)
SDEs, $\bar{\tau}$ satisfies (\ref{eq:ExpecedDurationPDE_General}).
However, as (\ref{eq:ExpecedDurationPDE_General}) is a parabolic
PDE, in the absence of initial (or terminal) conditions, PDE (\ref{eq:ExpecedDurationPDE_General})
alone is generally ill-posed. Indeed, in the general non-autonomous
non-periodic case, both the time and the initial/terminal condition
is part of the unknown. Indeed, if $\bar{\tau}(0,\cdot)$ is known,
then this implies we already know the expected exit time when the
system starts at time $s=0$, however it is an unknown to-be solved.
Resultantly, in some works such as \cite{BDE01}, it was assumed that
for sufficiently large $\tilde{T}$ that $u(x,y,\tilde{T})=0$ for
all $(x,y)\in D$. This is to impose some terminal condition of the
parabolic PDE. This however suggests that for all $(x,y)\in D$, the
expected exit time is a constant time $\tilde{T}$. While this assumption
may hold in certain systems, it does not hold in general. For example,
in the stochastic resonance problem that we shall discuss in Section
\ref{sec:NumericalSolutionOfPde} with Figure \ref{SDE_vs_PDE_plot},
(approximately) constant expected exit time only holds for a subdomain
rather than across the entire domain. In fact, in these papers, one
usually chooses $\tilde{T}>\sup\tau(x,y,t)$ by repeatedly running
SDE simulations to find a sufficiently large $\tilde{T}$ greater
than sample expected exit times. The results of this paper show that
one can avoid simulations altogether and only solve the periodic solution
of the PDE to attain the expected exit time. For time-periodic SDEs,
the boundary conditions issue are partially resolved by Lemma \ref{lem:TauIsPeriodic},
where initial and terminal conditions coincide, albeit unknown.
\end{rem}

\begin{rem}
It should be clear that for coefficients with non-trivial time-dependence,
the parabolic PDE (\ref{eq:ExpectedDurationPDE_Periodic}) would generally
imply that $\bar{\tau}(s,x)-\bar{\tau}(s',x)\neq(s-s')$ for $s\neq s'$.
That is, the difference in initial starting time does not imply the
same difference in expected time. This reinforce that initial time
generally plays a non-trivial role in the expected duration.
\end{rem}

As mentioned in the introduction, numerically solving PDE (\ref{eq:ExpectedDurationPDE_Periodic})
can be an appealing alternative to stochastic simulations of the expected
hitting time. We note further that solving (\ref{eq:ExpectedDurationPDE_Periodic})
solves the expected hitting time for all initial starting points.
On the other hand, direct simulation would (naively) require many
simulations for each starting point.

Assuming a priori that the expected exit time is finite, then we can
prove a converse of Theorem \ref{thm:ExpectedDurationPDE_Theorem}
via Dynkin's formula. In passing, this reassures that Theorem \ref{thm:ExpectedDurationPDE_Theorem}
is correct. 
\begin{prop}
Assume $\eta$ associated to $T$-periodic (\ref{eq:GeneralSDE_NonAuton})
has finite expectation. Then if (\ref{eq:ExpectedDurationPDE_Periodic})
has a solution $u$. Then $u(s,x)=\bar{\tau}(s,x)$. 
\end{prop}

\begin{proof}
Since $\eta$ is a stopping time and has finite expectation, by Itô's
and Dynkin's formula, then 
\[
\mathbb{E}^{s,x}\left[\varphi(\eta,X_{\eta})\right]=\varphi(s,x)+\mathbb{E}^{s,x}\left[\int_{s}^{\eta}(\partial_{t}+L(t))\varphi(t,X_{t})dt\right],\quad\varphi\in C_{0}^{1,2}(\mathbb{R}^{+}\times\mathbb{R}^{d}).
\]
Remark \ref{rem:Insolvability_of_general_nonauton} implies that there
does not generally exist a $u\in C^{1,2}(D)$ such that $(\partial_{s}+L(s))u(s,x)=-1$
and vanishes on $\partial D$ until we impose $T$-periodicity of
$u$. Therefore if such a $u$ exists, we have by (\ref{eq:Tau_Eta})
\begin{align*}
0 & =\mathbb{E}^{s,x}\left[u(\eta,X_{\eta})\right]=u(s,x)+\mathbb{E}^{s,x}\left[\int_{s}^{\eta}-1dt\right]=u(s,x)-\bar{\tau}(s,x).
\end{align*}
i.e. $u(s,x)=\bar{\tau}(s,x)$ and so the results follows.
\end{proof}

\section{\label{sec:PDE_Methods}Well-Posedness of Expected Duration PDE}

\subsection{\label{subsec:IBVPMethods}Fixed Point of an Initial Value Problem}

In this section, utilising classical results for the well-posedness
of initial value parabolic PDEs, we will show the existence of a unique
solution to the expected duration PDE (\ref{eq:ExpectedDurationPDE_Periodic})
for the associated $T$-periodic SDE. As mentioned in the introduction,
we solve (\ref{eq:ExpectedDurationPDE_Periodic}) with typical PDE
conditions rather than the stronger SDE\emph{ }conditions required
for the rigorous derivation of the PDE. This has the advantage of
a clearer exposition and key elements to solve the PDE.

In this subsection, we associate (\ref{eq:ExpectedDurationPDE_Periodic})
with an initial value boundary PDE problem and show that (\ref{eq:ExpectedDurationPDE_Periodic})
can be rewritten as a fixed point problem. We note however that (\ref{eq:ExpectedDurationPDE_Periodic}),
as an initial value problem, is a backward parabolic equation. Such
equations are known to be generally ill-posed in typical PDE spaces.
By reversing the time, we introduce a minus sign thus PDE is uniformly
elliptic and hence more readily solvable in typical function spaces.

We give a general uniqueness and existence result via a spectral result
of \cite{Hess_PeriodicBook} in $L^{p}(D)$. Specifically on $L^{2}(D)$,
we show that if the associated bilinear form is coercive then one
can apply a Banach fixed point argument to deduce the existence and
uniqueness. This yields a practical way to numerically compute the
desired solution.

To discuss the well-posedness of (\ref{eq:ExpectedDurationPDE_Periodic}),
we recall some standard Borel measurable function spaces. For any
$1\leq p<\infty$, we denote the Banach space $L^{p}(D)$ to be the
space of functions $f:D\rightarrow\mathbb{R}$ such that its norm
$\lVert f\rVert_{L^{p}(D)}:=\left(\int_{D}\lvert f(x)\rvert^{p}dx\right)^{1/p}<\infty$.
For $k\in\mathbb{N},$ we define as usual the Sobolev space $W^{k,p}(D)$
to contain all functions $f$ in which its norm $\lVert f\rVert_{W^{k,p}(D)}:=\left(\sum_{\lvert\beta\rvert\leq k}\lVert\partial^{\beta}f\rVert_{L^{p}(D)}^{p}\right)^{1/p}<\infty$.
We let $W_{0}^{k,p}(D)=\{f\in W^{k,p}(D)|f=0\quad\text{on }\partial D\}$.
For $p=2$, $L^{2}(D)$ and $H_{0}^{k}(D):=W_{0}^{k,2}(D)$ are Hilbert
spaces with inner-product $\langle f,g\rangle_{L^{2}(D)}:=\left(\int_{D}f(x)g(x)dx\right)^{1/2}$
and $\langle f,g\rangle_{H_{0}^{k}(D)}:=\sum_{\lvert\beta\rvert\leq k}\sum_{i=1}^{d}\langle\partial^{\beta}f,\partial^{\beta}g\rangle_{L^{2}(D)}$
respectively. Occasionally, we let $(H,\lVert\cdot\rVert_{H})$ denote
a generic Hilbert space. To avoid any possible confusion, we will
be verbose with the norms and inner-products.

We begin by fixing $1<p<\infty$ and define the time-reversed uniformly
elliptic operator associated to (\ref{eq:SpatialGenerator_fixed_s})
by

\begin{equation}
L_{R}(s):=L(T-s)=\sum_{i=1}^{d}b^{i}(T-s,x)\partial_{i}+\frac{1}{2}\sum_{i,j=1}^{d}a^{ij}(T-s,x)\partial_{ij}^{2},\quad s\in[0,T].\label{eq:ReversedOperator}
\end{equation}
Note that $\mathcal{D}(L_{R}(s))=W^{2,p}(D)\cap W_{0}^{1,p}(D)\subset L^{p}(D)$
for all $s\in[0,T]$. As mentioned, the initial boundary value problem
(IBVP) associated to (\ref{eq:ExpectedDurationPDE_Periodic}) is a
backward hence ill-posed in $L^{p}(D)$. Suppose that $u$ satisfies
(\ref{eq:ExpectedDurationPDE_Periodic}), consider the the time-reversed
solution $v(s,x)=u(T-s,x)$. Then $v$ satisfies 
\begin{equation}
\begin{cases}
\partial_{s}v-L_{R}(s)v=f, & \text{in }D_{T},\\
v=0 & \text{on }[0,T]\times\partial D,\\
v(0,\cdot)=v(T,\cdot),
\end{cases}\label{eq:ReversedPDE_Periodic}
\end{equation}
where $f\equiv1$. Clearly the solvability of (\ref{eq:ExpectedDurationPDE_Periodic})
is equivalent to (\ref{eq:ReversedPDE_Periodic}) up to time-reversal.
Hence, for the rest of the paper we focus on showing existence and
uniqueness of a solution to (\ref{eq:ReversedPDE_Periodic}).

Due to the general applicability of the methods presented in this
section, where possible, we retain a general inhomogeneous function
$f:[0,T]\rightarrow L^{p}(D)$. We expect that this generality benefits
some readers for solving similar problems.

The following IBVP associated to (\ref{eq:ReversedPDE_Periodic}),

\begin{equation}
\begin{cases}
(\partial_{s}-L_{R}(s))v=f, & \text{in }D_{T},\\
v=0, & \text{on [0,T]\ensuremath{\times\partial}D,}\\
v(0,\cdot)=v_{0} & \text{on }D,
\end{cases}\label{eq:Inhomo_Reversed_IBVP}
\end{equation}
is a ``forward'' parabolic equation and is readily solvable. We
say that $v$ is a generalised solution of (\ref{eq:Inhomo_Reversed_IBVP})
if $v\in C([0,T],W^{2,p}(D)\cap W_{0}^{1,p}(D))$, its derivative
$\frac{\partial v}{\partial s}\in C((0,T),L^{p}(D))$ exists and $v$
satisfies (\ref{eq:Inhomo_Reversed_IBVP}) in $L^{p}(D)$ \cite{Pazy_Semigroup,Amann_Parabolic,Daners_Medina}.
Consider also $\phi(s,x)$ satisfying the homogeneous PDE of (\ref{eq:Inhomo_Reversed_IBVP})
i.e. 
\begin{equation}
\begin{cases}
(\partial_{s}-L_{R}(s))\phi=0, & \text{in }(r,T)\times D,\\
\phi=0 & \mbox{on }[r,T]\times\partial D,\\
\phi(r,\cdot)=\phi_{r}, & \text{in }D.
\end{cases}\label{eq:TimeReversedHomo_IBVP}
\end{equation}
Given $\phi_{r}\in L^{p}(D)$, (\ref{eq:TimeReversedHomo_IBVP}) is
well-posed, we can define the evolution operator 
\begin{equation}
\Phi(r,s):L^{p}(D)\rightarrow W^{2,p}(D)\cap W_{0}^{1,p}(D),\quad r\leq s\leq T,\label{eq:EvolOpMapping}
\end{equation}
by 
\begin{equation}
\Phi(r,s)\phi_{r}:=\phi(s).\label{eq:EvolutionOperator}
\end{equation}
It is known that $\Phi(s,r)$ satisfies the semigroup property $\Phi(r,r)=Id$
and $\Phi(r,s)=\Phi(s,t)\Phi(r,s)$ for $r\leq s\leq t$. We refer
readers to \cite{Pazy_Semigroup} for regularity properties of $\Phi$.
When (\ref{eq:Inhomo_Reversed_IBVP}) is well-posed, it is well-known
that by a variation of constants or Duhamel's formula \cite{Amann_Parabolic,Daners_Medina,Pazy_Semigroup},
the solution to inhomogeneous problem (\ref{eq:Inhomo_Reversed_IBVP})
satisfies 
\begin{equation}
v(s)=\Phi(r,s)v_{r}+\int_{r}^{s}\Phi(r',s)f(r')dr'.\label{eq:InhomoIntegralEquation}
\end{equation}
It is well-known that if Condition A1 holds and $f\in C^{\gamma}(0,T;L^{p}(D))$
for some $\gamma\in(0,1)$, then (\ref{eq:Inhomo_Reversed_IBVP})
is well-posed \cite{Pazy_Semigroup,Amann_Parabolic}. Furthermore,
we can define the solution operator after one period $\mathcal{A}:L^{p}(D)\rightarrow W^{2,p}(D)\cap W_{0}^{1,p}(D)$
by 
\begin{equation}
\mathcal{A}\varphi:=\Phi(0,T)\varphi+\int_{0}^{T}\Phi(r,T)f(r)dr.\label{eq:OperatorA}
\end{equation}
We discuss further conditions for regular solutions. Theorem 24.2
of \cite{Daners_Medina} employed Schauder estimates and Sobolev embedding
to show that if $p>d/2$, $\partial D\in C^{2}(\mathbb{R}^{d-1})$
then the solution to IBVP (\ref{eq:Inhomo_Reversed_IBVP}) with initial
condition $v_{0}\in W_{0}^{2,p}(D)$ satisfies the following regularity
\begin{equation}
v\in C(\bar{D}_{T})\cap C^{1+\frac{\theta}{2},2+\theta}((0,T]\times\bar{D}).\label{eq:RegularSoln}
\end{equation}
Furthermore, if $d<p<\infty$, by Sobolev embedding, then $v\in C^{\frac{1+\xi}{2},1+\xi}(\bar{D}_{T})\cap C^{1+\frac{\theta}{2},2+\theta}((0,T]\times\bar{D})$
for some $\xi\in(0,1)$, see \cite{Hess_PeriodicBook}. Thus, we write
our first existence and uniqueness result. 
\begin{prop}
\label{prop:ExistenceUniquness_PeriodicSoln} Assume Condition A1
holds. Assume that $d<p<\infty$, $\partial D\in C^{2}(\mathbb{R}^{d-1})$
and $f\in C^{\gamma}(0,T;L^{p}(D))$ for some $\gamma\in(0,1)$. Then
there exists a unique regular solution satisfying (\ref{eq:ReversedPDE_Periodic}).
Moreover, if $f\neq0$ then the solution is non-trivial.
\end{prop}

\begin{proof}
Since $f\in C^{\gamma}(0,T;L^{p}(D))$, by Condition A1, IBVP (\ref{eq:Inhomo_Reversed_IBVP})
is well-posed for any $v_{0}\in L^{p}(D)$. Hence the evolution operator
$\Phi$ defined by (\ref{eq:EvolutionOperator}) is well-defined.
In general, to solve $T$-periodic PDE (\ref{eq:ReversedPDE_Periodic}),
by Duhamel's formula (\ref{eq:InhomoIntegralEquation}), one wishes
to find existence and uniqueness of a $v_{0}\in L^{p}(D)$ such that
\begin{equation}
v_{0}=\mathcal{A}v_{0}.\label{eq:ImposedPeriodForParabolicPDE}
\end{equation}
For initial conditions in $W_{0}^{2,p}(D)$, by rearranging from (\ref{eq:OperatorA}),
we have 
\begin{equation}
(I-\Phi(0,T))v_{0}=\int_{0}^{T}\Phi(r,T)f(r)dr,\label{eq:(I-U)v0}
\end{equation}
where $\Phi(0,s):W_{0}^{2,p}(D)\rightarrow W_{0}^{2,p}(D)$ and $I:W_{0}^{2,p}(D)\rightarrow W_{0}^{2,p}(D)$
is the identity operator. With the current conditions, via Krein-Rutman
theorem, it was shown in \cite{Hess_PeriodicBook} that $\lambda=\rho(\Phi(0,T))\in(0,1)$,
where $\lambda$ denotes the spectral radius of $\Phi(0,T)$. This
implies that $1$ is in the resolvent i.e. $(I-\Phi(0,T)):W_{0}^{2,p}(D)\rightarrow W_{0}^{2,p}(D)$
is invertible. It follows that 
\begin{equation}
v_{0}=(I-\Phi(0,T))^{-1}\int_{0}^{T}\Phi(r,T)f(r)dr,\label{eq:Solution_by_inverting_(Id-U)-1}
\end{equation}
uniquely solves (\ref{eq:ImposedPeriodForParabolicPDE}). By Sobolev
embedding, 
\begin{equation}
v(s,\cdot)=\Phi(0,s)v_{0}+\int_{0}^{s}\Phi(r,s)f(r)dr,\quad s\in(0,T],\label{eq:SolnToPeriodicPDE_Variation_of_Const}
\end{equation}
is a regular solution to (\ref{eq:ReversedPDE_Periodic}). It is easy
to see that (\ref{eq:ReversedPDE_Periodic}) does not admit trivial
solutions since ($D$ is non-empty and) $v\equiv0$ cannot satisfy
(\ref{eq:ReversedPDE_Periodic}) for $f\neq0$. 
\end{proof}
As noted in \cite{Daners_Medina}, via the semigroup property, one
can approximate $\Phi(0,T)\simeq\prod_{n=0}^{N-1}\Phi(t_{n},t_{n+1})$
for $0=t_{0}<t_{1}<...<t_{N}=T$. Hence one can approximate the inverse
in (\ref{eq:Solution_by_inverting_(Id-U)-1}) by 
\begin{equation}
(I-\Phi(0,T))^{-1}\simeq(I-\prod_{n=0}^{N-1}\Phi(t_{n},t_{n+1}))^{-1}.\label{eq:InvertingProdMatrix}
\end{equation}
We note however computing (\ref{eq:InvertingProdMatrix}) is generally
computationally expensive.

We can gain more from (\ref{eq:Solution_by_inverting_(Id-U)-1}).
We recall the weak maximum principle: if the solution is regular and
$f\geq0$, then 
\begin{equation}
\min_{(s,x)\in\bar{D}_{T}}v(s,x)=\min_{x\in\bar{D}}v_{r}(x)\label{eq:MinimumPrinciple}
\end{equation}
holds. We have seen that, by (\ref{eq:(I-U)v0}), the existence and
uniqueness of $v_{0}\in L^{2}(D)$ satisfying (\ref{eq:ImposedPeriodForParabolicPDE})
requires the invertibility of $I-\Phi(0,T)$. By von Neumann series,
we have 
\[
(I-\Phi(0,T))^{-1}=\sum_{k=0}^{\infty}\Phi^{k}(0,T),
\]
where $\Phi^{k}(0,T)$ denotes the composition of the operator $\Phi(0,T)$.

It is well-known that parabolic PDEs experience parabolic smoothing
(see e.g. \cite{Pazy_Semigroup,Evans_PDEs}) i.e. the solution of
parabolic equations are as smooth as the coefficients and initial
data. For example if $p>d/2$ and $f\in C^{\gamma}(0,T;W^{2,p}(D))$,
then $\Phi(s,t)f$ is a regular solution by (\ref{eq:RegularSoln}).
Moreover, if $f\geq0$, by the maximum principle, $\Phi(s,t)f\geq0$
for all $0\leq s\leq t\leq T$. It follows that $\mathcal{I}:=\int_{0}^{T}\Phi(r,T)f(r)dr\geq0$
and $\Phi^{k}(0,T)\mathcal{I}\geq0$ for all $k\in\mathbb{N}$. Moreover,
it follows from (\ref{eq:Solution_by_inverting_(Id-U)-1}) that 
\begin{align}
v_{0} & =\sum_{k=0}^{\infty}\Phi^{k}(0,T)\mathcal{I}\geq0,\label{eq:v0_is_nonnegative}
\end{align}
i.e. the solution to (\ref{eq:ImposedPeriodForParabolicPDE}) is non-negative.
Furthermore, if the coefficients and $f$ are smooth then condition
$p>d/2$ can be dropped and the same conclusion holds with a smooth
solution \cite{Pazy_Semigroup}. In particular, since $1\in C^{\infty}((0,T)\times D)$
is non-negative, this aligns with physical reality that expected duration
time $\bar{\tau}(0,\cdot)=v_{0}$ indeed is non-negative.

To gain further insight into solving (\ref{eq:ReversedPDE_Periodic})
from both a theoretically and computational viewpoint, we progress
our study with Hilbert spaces i.e. $p=2$ and forego some of the regularity
gained from Sobolev embedding e.g. (\ref{eq:RegularSoln}). The following
approach allows us to study (\ref{eq:ReversedPDE_Periodic}).

We start with a standard framework to deduce the existence and uniqueness
of (\ref{eq:ReversedPDE_Periodic}) on the Hilbert space $L^{2}(D)$.
For convenience, we define the bilinear form $B_{R}:H_{0}^{1}(D)\times H_{0}^{1}(D)\rightarrow\mathbb{R}$
associated to $-L_{R}$ defined by 
\begin{equation}
B_{R}[\varphi,\psi;s]=-\sum_{i=1}^{d}\int_{D}\tilde{b}^{i}(T-s,x)\partial_{i}\varphi(x)\psi(x)dx+\frac{1}{2}\sum_{i,j=1}^{d}\int_{D}a^{ij}(T-s,x)\partial_{i}\varphi(x)\partial_{j}\psi(x)dx,\label{eq:ReveresdBilinear}
\end{equation}
where $\tilde{b}^{i}(s,x)=b^{i}(s,x)+\sum_{j=1}^{d}\partial_{j}a^{ij}(s,x)$
for each $1\le i\leq d$. We recall that a bilinear form $B_{R}:H_{0}^{1}(D)\times H_{0}^{1}(D)\rightarrow\mathbb{R}$
is coercive if there exists a constant $\alpha>0$ such that 
\begin{equation}
B_{R}[\varphi,\varphi;s]\geq\alpha\left\lVert \varphi\right\rVert _{H_{0}^{1}(D)}^{2},\quad\varphi\in H_{0}^{1}(D),s\in[0,T].\label{eq:Coercive}
\end{equation}

Assuming coercivity, we give the following existence and uniqueness
theorem to (\ref{eq:ReversedPDE_Periodic}).
\begin{thm}
\label{thm:ExistenceUniquenessResult_Coercive} Assume that $a^{ij},b^{i}\in L^{\infty}(D_{T})$
and $a(\cdot,\cdot)$ satisfies uniformly elliptic condition (\ref{eq:UniformlyElliptic})
and furthermore (\ref{eq:ReveresdBilinear}) is coercive for $s\in[0,T]$.
Then for any $f\in L^{2}(0,T;L^{2}(D))$, there exists a unique solution
$v\in C([0,T],H_{0}^{1}(D))$to (\ref{eq:ReversedPDE_Periodic}).
If $f\neq0$, then the solution is non-trivial. 
\end{thm}

\begin{proof}
It is well-known (e.g. \cite{Evans_PDEs}) that there exists a unique
weak solution $v$ to the IBVP (\ref{eq:Inhomo_Reversed_IBVP}) i.e.
$v\in C([r,T];L^{2}(D))\cap L^{2}(r,T;H_{0}^{1}(D))$ such that $v(r)=v_{r}$,
$\partial_{s}v\in L^{2}(r,T;H^{-1}(D))$ and for almost every $s\in[r,T]$,
\begin{equation}
\langle\partial_{s}v(s),\varphi\rangle_{H^{-1}(D)\times H_{0}^{1}(D)}+B_{R}[v,\varphi;s]=\langle f(s),\varphi\rangle_{H^{-1}(D)\times H_{0}^{1}(D)},\quad\varphi\in H_{0}^{1}(D)\text{,}\label{eq:WeakFormulation}
\end{equation}
where $H^{-1}(D)$ is the space of linear functionals of the subspace
$H_{0}^{1}(D)$ on $L^{2}(D)$ and $\langle\cdot,\cdot\rangle_{H^{-1}(D)\times H_{0}^{1}(D)}:H^{-1}(D)\times H_{0}^{1}(D)\rightarrow\mathbb{R}$
denotes the duality pairing between $H^{-1}(D)$ and $H_{0}^{1}(D)$.
To prove our result, it is sufficient to assume $f\in L^{2}(D)$.
To cast (\ref{eq:ImposedPeriodForParabolicPDE}) in terms of a self-mapping,
consider $\bar{\Phi}(0,T):L^{2}(D)\rightarrow L^{2}(D)$ as the operator
$\Phi(0,T)$ with its range enlarged to $L^{2}(D)$ and define $\mathcal{\bar{A}}:L^{2}(D)\rightarrow L^{2}(D)$
by 
\begin{equation}
\mathcal{\bar{A}}\varphi:=\bar{\Phi}(0,T)\varphi+\int_{0}^{T}\bar{\Phi}(r,T)f(r)dr.\label{eq:Operator_A_for_FP}
\end{equation}
We show there exists a unique fixed point of operator $\bar{\mathcal{A}}$.
By Banach fixed point theorem, it suffices to show $\mathcal{\bar{A}}$
is a contraction on $L^{2}(D)$. Observe that this is sufficient provided
$\bar{\Phi}(0,T)$ is a contraction mapping on $L^{2}(D)$ since 
\begin{align*}
\left\lVert \mathcal{\bar{A}}\varphi-\bar{\mathcal{A}}\psi\right\rVert _{L^{2}(D)} & =\left\lVert \bar{\Phi}(0,T)(\varphi-\psi)\right\rVert _{L^{2}(D)}\leq\lVert\bar{\Phi}(0,T)\rVert\left\lVert \varphi-\psi\right\rVert _{L^{2}(D)},\quad\varphi,\psi\in L^{2}(D).
\end{align*}
In fact, we show that $\bar{\Phi}(0,s)$ is a contraction for any
$s>0$.

From (\ref{eq:EvolOpMapping}), for initial condition $\phi_{0}\in L^{2}(D)$,
the homogeneous solution, $\phi(s)\in H^{2}(D)\cap H_{0}^{1}(D)$
satisfies (\ref{eq:TimeReversedHomo_IBVP}). Then from (\ref{eq:WeakFormulation}),
one has by coercivity 
\begin{align*}
0 & =\langle\partial_{s}\phi(s),\phi(s)\rangle+B_{R}[\phi(s),\phi(s);s]\\
 & \geq\frac{1}{2}\frac{d}{ds}\left\lVert \phi(s)\right\rVert _{L^{2}(D)}^{2}+\alpha\left\lVert \phi(s)\right\rVert _{H_{0}^{1}(D)}^{2}\\
 & \geq\frac{1}{2}\frac{d}{ds}\left\lVert \phi(s)\right\rVert _{L^{2}(D)}^{2}+\alpha\left\lVert \phi(s)\right\rVert _{L^{2}(D)}^{2}.
\end{align*}
Gronwall's inequality then yields 
\begin{align*}
\left\lVert \phi(s)\right\rVert _{L^{2}(D)}^{2} & \leq e^{-2\alpha s}\left\lVert \phi_{0}\right\rVert _{L^{2}(D)}^{2},\quad s\geq0.
\end{align*}
Hence indeed 
\begin{align}
\left\lVert \bar{\Phi}(0,s)\right\rVert := & \sup_{\phi_{0}\in L^{2}(D)}\frac{\left\lVert \bar{\Phi}(0,s)\phi_{0}\right\rVert _{L^{2}(D)}}{\left\lVert \phi_{0}\right\rVert _{L^{2}(D)}}\leq e^{-\alpha s}<1,\quad s>0,\label{eq:U_is_contraction_on_L2}
\end{align}
i.e. $\bar{\Phi}(0,s)$ is a contraction on $L^{2}(D)$. Therefore
there exists a unique $v_{0}\in L^{2}(D)$ satisfying (\ref{eq:Operator_A_for_FP}).
Since $\mathcal{A:}L^{2}(D)\rightarrow H_{0}^{1}(D)\subsetneq L^{2}(D)$,
then by the right hand side of (\ref{eq:ImposedPeriodForParabolicPDE}),
it is easy to deduce that $v_{0}\in H_{0}^{1}(D)$. Define $v$ by
(\ref{eq:SolnToPeriodicPDE_Variation_of_Const}), then $v\in C([0,T],H_{0}^{1}(D))$
is the unique solution to (\ref{eq:ReversedPDE_Periodic}). Lastly,
if $0\neq f\in L^{2}(D)$, then $v$ is non-trivial.
\end{proof}
Theorem \ref{thm:ExistenceUniquenessResult_Coercive} offers not only
a theoretical existence and uniqueness result on the solution to (\ref{eq:ReversedPDE_Periodic}),
by Banach fixed point, Theorem \ref{thm:ExistenceUniquenessResult_Coercive}
immediately offers an iterative numerical approach to the solution.
To numerically computing the next Banach fixed point iterate, one
only requires to solve a IBVP for the parabolic PDE. Compared to (\ref{eq:InvertingProdMatrix}),
there are well-established numerical schemes for parabolic PDEs with
known order of convergences.

We remark that coercivity is actually stronger than required. In the
proof of Theorem \ref{thm:ExistenceUniquenessResult_Coercive}, it
is sufficient that $B[\varphi,\varphi;s]\geq\alpha\lVert\varphi\rVert_{L^{2}(D)}^{2}$.
We give an example where coercivity is shown. We consider the example
of a one-dimensional Brownian motion with periodic drift. 
\begin{example}
\label{exa:BasicCoerciveExample}Let $S\in C^{1}(\mathbb{R}^{+})$
be a $T$-periodic function and $\sigma\neq0$ and consider the one-dimensional
$T$-periodic SDE 
\[
dX_{t}=S(t)dt+\sigma dW_{t},
\]
on some bounded interval $D$. Clearly Condition A2 is satisfied.
By Theorem \ref{thm:ExistenceUniquenessResult_Coercive}, it is sufficient
to show the associated (time-reversed) bilinear form 
\[
B_{R}[\varphi,\psi;s]=-\int_{D}S(T-s)\partial_{x}\varphi(x)\psi(x)dx+\frac{\sigma^{2}}{2}\int_{D}\partial_{x}\varphi(x)\partial_{x}\psi(x)dx,\quad\varphi,\psi\in H_{0}^{1}(D),
\]
is coercive. This is obvious by an integration by parts with vanishing
boundaries and applying the Poincaré inequality 
\begin{align*}
B_{R}[\varphi,\varphi;s] & =-\frac{S(T-s)}{2}\int_{D}\partial_{x}(\varphi^{2}(x))dx+\frac{\sigma^{2}}{2}\lVert\partial_{x}\varphi\rVert_{L^{2}(D)}^{2}\\
 & \geq-\frac{S(T-s)}{2}\varphi^{2}(x)\rvert_{\partial D}+\frac{\sigma^{2}}{4}\left\lVert \partial_{x}\varphi\right\rVert _{L^{2}(D)}^{2}+\frac{\sigma^{2}}{4C_{D}}\left\lVert \varphi\right\rVert _{L^{2}(D)}^{2}\\
 & \geq\alpha\lVert\varphi\rVert_{H_{0}^{1}(D)}^{2},
\end{align*}
where $C_{D}$ denotes the Poincaré constant for the domain $D$ such
that $\left\lVert \varphi\right\rVert _{L^{2}(D)}^{2}\leq C_{D}\left\lVert \partial_{x}\varphi\right\rVert _{L^{2}(D)}^{2}$
and $\alpha=\min(\frac{\sigma^{2}}{4},\frac{\sigma^{2}}{4C_{D}})>0$.
Hence by Theorem \ref{thm:ExistenceUniquenessResult_Coercive}, there
exists a unique solution to (\ref{eq:ExpectedDurationPDE_Periodic}). 
\end{example}

\subsection{Convex Optimisation}

In Section \ref{subsec:IBVPMethods}, we showed that if the bilinear
form associated to the PDE is coercive, then Theorem \ref{thm:ExistenceUniquenessResult_Coercive}
yields a unique solution to (\ref{eq:ReversedPDE_Periodic}). However,
in general, coercivity of the associated bilinear form can be difficult
to verify. Instead, we now seek to solve (\ref{eq:ReversedPDE_Periodic})
by casting it as a convex optimisation problem with a natural cost
functional. Convex optimisation has been a standard method to study
solutions of elliptic PDEs.

In the presented convex optimisation framework, we show that the unique
minimiser of the cost functional is a solution to (\ref{eq:ReversedPDE_Periodic}).
In this approach, coercivity of the functional holds almost immediately,
provided the maximum principle holds. When the maximum principle holds,
it follows that the solution is a classical/strong solution as opposed
to the weak solution given in Theorem \ref{thm:ExistenceUniquenessResult_Coercive}.Furthermore,
we show that the convex optimisation problem can be implemented readily
by standard gradient methods.

We begin with a standard convex optimisation result on Hilbert spaces.
Let $(H,\lVert\cdot\rVert_{H}$) be a Hilbert space, $\mathscr{C}\subseteq H$
be a closed convex subset and $F:H\rightarrow\mathbb{R}$ be a functional.
The functional $F$ is said to be norm-like (or coercive) over $\mathscr{C}$
if 
\[
F(\varphi)\rightarrow\infty,\quad\text{as }\lVert\varphi\rVert_{H}\rightarrow\infty,\quad\varphi\in\mathscr{C}.
\]
The functional $F$ is Gateaux differentiable at $\varphi\in H$ if
for any $\phi_{0}\in H$, the directional derivative of $F$ at $\varphi$
in the direction $\phi_{0}$, denoted by $DF(\varphi)(\phi_{0})$,
given by 
\begin{equation}
DF(\varphi)(\phi_{0})=\lim_{\epsilon\rightarrow0}\frac{F(\varphi+\epsilon\phi_{0})-F(\varphi)}{\epsilon}\label{eq:DirectionDerivDefn}
\end{equation}
exists. By definition, the gradient of $F$ is $\frac{\delta F}{\delta\varphi}\in L^{2}(D)$
satisfying
\begin{equation}
\langle\frac{\delta F}{\delta\varphi},\phi_{0}\rangle=DF(\varphi)(\phi_{0}),\quad\phi_{0}\in H.\label{eq:GradFunctional}
\end{equation}
We shall use the following standard result from convex optimisation
theory (see e.g. \cite{ConvexAnalysis,TroltzschOptimalControlPDE}). 
\begin{lem}
\label{lem:ConvexLemma} Let $H$ be a Hilbert space and $\mathscr{C}\subseteq H$
be a closed convex set. Let $F:H\rightarrow\mathbb{R}$ be a functional
such that $F$ is convex and norm-like over $\mathscr{C}$. Assume
further that $F$ is a (lower semi)continuous and bounded from below.
Then there exists at least one $v_{0}\in\mathscr{C}$ such that $F(v_{0})=\inf_{\varphi\in\mathscr{C}}F(\varphi)$.
If $F$ is Gateaux differentiable, then for any such $v_{0}$, $DF(v_{0})(\cdot)=0$.
If $F$ is strictly convex then $v_{0}$ is unique. 
\end{lem}

We now focus specifically on using Lemma \ref{lem:ConvexLemma} to
solve (\ref{eq:ReversedPDE_Periodic}). Recall that if Condition A1
holds then (\ref{eq:Inhomo_Reversed_IBVP}) is well-posed. We then
associate to (\ref{eq:Inhomo_Reversed_IBVP}) the natural cost functional
$F:L^{2}(D)\rightarrow\mathbb{R}$ defined by 
\begin{equation}
F(\varphi)=\frac{1}{2}\lVert\mathcal{A}\varphi-\varphi\rVert_{L^{2}(D)}^{2}=\frac{1}{2}\int_{D}[(\mathcal{A}\varphi)(x)-\varphi(x)]^{2}dx,\label{eq:CostFunctional}
\end{equation}
where $\mathcal{A}$ is given by (\ref{eq:OperatorA}). This functional
is natural to our periodic problem because if there exists $v_{0}\in L^{2}(D)$
which minimises the functional to zero, it is a solution to (\ref{eq:ReversedPDE_Periodic})
i.e. 
\[
F(v_{0})=0\quad\iff\quad\mathcal{A}v_{0}=v_{0},
\]
i.e. $v_{0}$ solves (\ref{eq:ImposedPeriodForParabolicPDE}) and
therefore is a (possibly weak) solution to (\ref{eq:ReversedPDE_Periodic}).

Optimisation briefly aside, we recommend using the cost functional
$F$ to quantify the convergence of the Banach iterates of Theorem
\ref{thm:ExistenceUniquenessResult_Coercive}.

In order to apply Lemma \ref{lem:ConvexLemma} on $F$, we recall
some properties associated to linear parabolic PDEs. Suppose that
(\ref{eq:Inhomo_Reversed_IBVP}) is well-posed. Since PDE (\ref{eq:Inhomo_Reversed_IBVP})
is linear, by the superposition principle, $\Phi(\cdot,\cdot)$ is
a linear operator, i.e. $\Phi(s,t)[\lambda_{1}\phi_{1}+\lambda_{2}\phi_{2}]=\lambda_{1}\Phi(s,t)\phi_{1}+\lambda_{2}\Phi(s,t)\phi_{2}$.
However, due to the inhomogeneous term, observe that $\mathcal{A}$
is not linear. Instead, if $\lambda_{1},\lambda_{2}\geq0$ such that
$\lambda_{1}+\lambda_{2}=1$ then 
\begin{align}
\mathcal{A}(\lambda_{1}\varphi+\lambda_{2}\psi) & =\Phi(0,T)[\lambda_{1}\varphi+\lambda_{2}\psi]+(\lambda_{1}+\lambda_{2})\int_{0}^{T}\Phi(r,T)f(r)dr=\lambda_{1}\mathcal{A}\varphi+\lambda_{2}\mathcal{A}\psi.\label{eq:Linearity_of_A}
\end{align}
Since $\bar{\tau}$ is non-negative, we consider 
\[
\mathscr{C}(D):=\{\varphi\in H_{0}^{2}(D)|\varphi\geq0\}.
\]
It is easy to verify that $\mathscr{C}(D)$ is a closed convex Hilbert
subspace of $L^{2}(D)$. 
\begin{thm}
\label{thm:FunctionalUniqueSoln}Let Condition A1 hold, $f\in C^{\gamma}(0,T;L^{2}(D))$
for some $\gamma\in(0,1)$, $f\geq0$ and $d\leq3$. Let $F:\mathscr{C}(D)\subset L^{2}(D)\rightarrow\mathbb{R}$
be defined by (\ref{eq:CostFunctional}) . Then there exists a unique
$v_{0}\in\mathscr{C}(D)$ minimising $F$. 
\end{thm}

\begin{proof}
Since Condition A1 holds, the IBVP (\ref{eq:Inhomo_Reversed_IBVP})
is well-posed. Hence the operators $\mathcal{A}$ and $\Phi(s,t)$
and thus $F$ are all well-defined. We show that $F$ satisfies the
assumptions of Lemma \ref{lem:ConvexLemma}. Obviously $F\geq0$ and
hence bounded from below. By the well-posedness of (\ref{eq:Inhomo_Reversed_IBVP}),
it is clear that $\varphi\rightarrow\mathcal{\mathcal{A\varphi}}$
and moreover $\varphi\rightarrow\mathcal{\mathcal{A\varphi}}-\varphi$
are continuous from $L^{2}(D)$ to $L^{2}(D)$. It follows that $F$
is continuous. Utilising the strong convexity of the quadratic function
and (\ref{eq:Linearity_of_A}), we show the strong convexity of $F$:
for any $\lambda\in(0,1)$ and $\varphi,\psi\in\mathscr{C}(D)$, we
have that 
\begin{align*}
F(\lambda\varphi+(1-\lambda)\psi) & =\frac{1}{2}\int_{D}[\mathcal{A}(\lambda\varphi+(1-\lambda)\psi)-(\lambda\varphi+(1-\lambda)\psi)]^{2}dx\\
 & =\frac{1}{2}\int_{D}[\lambda\left(\mathcal{A}\varphi-\varphi\right)+(1-\lambda)(\mathcal{A}\psi-\psi)]^{2}dx\\
 & <\frac{1}{2}\int_{D}\left[\lambda\left(\mathcal{A}\varphi-\varphi\right)^{2}+(1-\lambda)(\mathcal{A}\psi-\psi)^{2}\right]dx\\
 & =F(\varphi)+(1-\lambda)F(\psi).
\end{align*}
Since $d\leq3$, by Sobolev embedding and Schauder estimates, it follows
from (\ref{eq:RegularSoln}) (as $\frac{d}{2}<p=2$) that for any
$\varphi\in\mathscr{C}(D)\subset H^{2}(D)$ and $f\geq0$, the solution
to (\ref{eq:Inhomo_Reversed_IBVP}) with initial condition $\varphi$
is regular. Therefore, the maximum principle (\ref{eq:MinimumPrinciple})
applies. Hence together with the homogeneous Dirichlet boundary conditions,
it follows that $\mathcal{A}\varphi\geq0$. Therefore, for any $x\in D$
and $\epsilon\in(0,1)$, Young's inequality yields that 
\begin{align*}
F(\varphi) & =\frac{1}{2}\int_{D}((\mathcal{A}\varphi)^{2}-2(\mathcal{A}\varphi)\varphi+\varphi^{2})dx\\
 & \geq\frac{1}{2}\int_{D}((1-\epsilon)\varphi^{2}+(1-\epsilon^{-1})(\mathcal{A}\varphi)^{2})dx\\
 & =\frac{1-\epsilon}{2}\lVert\varphi\rVert_{L^{2}(D)}^{2}+\frac{1-\epsilon^{-1}}{2}\lVert\mathcal{A}\varphi\rVert_{L^{2}(D)}^{2}.
\end{align*}
Hence it follows that $F(\varphi)\rightarrow\infty$ as $\lVert\varphi\rVert_{L^{2}(D)}\rightarrow\infty$.
Then Lemma \ref{lem:ConvexLemma} yields a unique solution $v_{0}\in\mathscr{C}(D)$
minimising $F$. 
\end{proof}
In the following proposition, we derive an expression for the directional
derivative $DF(\varphi)(\phi_{0})$. While it is then straightforward
to apply the maximum principle to show that $DF(\varphi)(\cdot)$
is a linear continuous operator to deduce existence and uniqueness
of the gradient (via Riesz representation theorem), we employ numerical
analysts' adjoint state method (see e.g. \cite{AdjointIntro,AdjointRigour,AdjointStateMethodEfficient,AdjointStateMethodReview})
to attain an expression for the gradient directly. From a numerical
perspective, the gradient allows us to apply gradient methods to iteratively
minimise $F$. Numerically, we note that it is not necessary to use
adjoint state method to compute the gradient. However, it is well-known
that adjoint state method is (generally) computationally efficient
see e.g. \cite{AdjointStateMethodEfficient}. It is noted that comparing
to Banach fixed point iterates of Theorem \ref{thm:ExistenceUniquenessResult_Coercive},
the adjoint state method is computationally less efficient because
a pair of IBVP is required to be solved rather than one.

To employ the adjoint state method, we recall that $L^{*}(s)$ is
the Fokker-Planck operator given by (\ref{eq:FokkerPlanckGeneral}).
Akin to (\ref{eq:FokkerPlanck_AbsorbingBoundaries}), if Condition
A2 holds, then 
\begin{equation}
\begin{cases}
\partial_{s}w=L^{*}(s)w & \text{in }D_{T},\\
w=0 & \text{on }[0,T]\times\partial D\\
w(0,\cdot)=w_{0}, & \text{on }D,
\end{cases},\label{eq:PDE_for_w}
\end{equation}
is well-posed for any $w_{0}\in L^{2}(D)$\cite{Pazy_Semigroup,Evans_PDEs,Amann_Parabolic}.
Hence, akin to (\ref{eq:TimeReversedHomo_IBVP}) and (\ref{eq:EvolOpMapping}),
we define the evolution operator $W$ for (\ref{eq:PDE_for_w}) i.e.
\begin{equation}
w(s,\cdot)=W(0,s)w_{0},\quad s\geq0,\label{eq:OperatorW}
\end{equation}
where $W(0,s):L^{2}(D)\rightarrow H^{2}(D)\cap H_{0}^{1}(D)$.

The following proposition was inspired by \cite{AmbroseWilkening_NumericalGradientMethod,BristeauGlowPer_NumericalGradient}
in employing the adjoint state method for periodic solutions of the
Benjamin-Ono and autonomous evolution equations respectively. It plays
a significant role to prove Theorem \ref{thm:ConvexThm}.
\begin{prop}
\label{prop:FunctionalDerivatives}Let Condition A2 hold and $F$
be defined by (\ref{eq:CostFunctional}). Then for any $\varphi\in L^{2}(D)$,
we have the expressions for the directional derivative 
\begin{equation}
DF(\varphi)(\phi_{0})=\int_{D}(\mathcal{A}\varphi-\varphi)(\Phi(0,T)\phi_{0}-\phi_{0})dx,\quad\phi_{0}\in L^{2}(D),\label{eq:DirectionalDerivOfF}
\end{equation}
and the gradient 
\begin{equation}
\frac{\delta F}{\delta\varphi}=W(0,T)w_{0}-w_{0},\label{eq:Gradient_Functional_Expression}
\end{equation}
with initial condition $w_{0}=\mathcal{A}\varphi-\varphi$.
\end{prop}

\begin{proof}
Utilising the linearity properties of $\mathcal{A}$ and $\Phi(0,T)$,
from (\ref{eq:OperatorA}) and (\ref{eq:DirectionDerivDefn}), we
have 
\begin{align*}
DF(\varphi)(\phi_{0}) & =\lim_{\epsilon\rightarrow0}\frac{1}{2}\int_{D}\frac{(\mathcal{A}(\varphi+\epsilon\phi_{0})-(\varphi+\epsilon\phi_{0}))^{2}-(\mathcal{A}\varphi-\varphi)^{2}}{\epsilon}dx\\
 & =\lim_{\epsilon\rightarrow0}\frac{1}{2}\int_{D}\frac{((\mathcal{A}\varphi-\varphi)+\epsilon(\Phi(0,T)\phi_{0}-\phi_{0}))^{2}-(\mathcal{A}\varphi-\varphi)^{2}}{\epsilon}dx.
\end{align*}
Hence (\ref{eq:DirectionalDerivOfF}) follows by collecting terms
and taking the limit. We now wish to find $\frac{\delta F}{\delta\varphi}\in L^{2}(D)$
such that

\[
DF(\varphi)(\phi_{0})=\langle\frac{\delta F}{\delta\varphi},\phi_{0}\rangle=\int_{D}\frac{\delta F}{\delta\varphi}(x)\phi_{0}(x)dx,\quad\phi_{0}\in L^{2}(D).
\]
To compute the gradient, consider $w_{R}(s,x)$ satisfying the adjoint
equation of PDE (\ref{eq:TimeReversedHomo_IBVP})

\begin{equation}
\begin{cases}
-\partial_{s}w_{R}=L_{R}^{*}(s)w_{R} & \text{in }D_{T},\\
w_{R}=0, & \text{on }[0,T]\times\partial D,\\
w_{R}(T,\cdot)=w_{T}, & \text{on }D.
\end{cases}\label{eq:PDE_for_wR}
\end{equation}
for some terminal condition $w_{T}\in L^{2}(D)$. Note that (\ref{eq:PDE_for_wR})
is a backward equation. However, as terminal conditions are provided,
(\ref{eq:PDE_for_wR}) is well-posed provided Condition A2 are satisfied.
This contrasts to the backward equation associated to (\ref{eq:ExpectedDurationPDE_Periodic})
as a IBVP with initial conditions. For clarity, we introduce $w(s,\cdot)=w_{R}(T-s,\cdot)$.
Then it is clear that $w$ satisfies (\ref{eq:PDE_for_w}), since
$L_{R}^{*}(T-s)=L^{*}(T-(T-s))=L^{*}(s)$ by (\ref{eq:ReversedOperator}).
In this form, it is clear that (\ref{eq:PDE_for_w}) and equivalently
(\ref{eq:PDE_for_wR}) are well-posed provided Condition A2 is satisfied.
With the repeated time-reversal, $w(s,x)$ is understood to solve
the Fokker-Planck equation forward in time.

Let $\phi$ be the homogeneous solution satisfying (\ref{eq:TimeReversedHomo_IBVP})
with initial conditions $\phi(0,\cdot)=\phi_{0}$. By multiplying
$\phi$ by $w_{R}$ and integrating by parts over $D_{T}$, we have
by (\ref{eq:PDE_for_wR}) 
\begin{align*}
0 & =\int_{0}^{T}\int_{D}(\partial_{s}\phi(s,x)-L_{R}(s)\phi(s,x))\cdot w_{R}(s,x)dxds\\
 & =\int_{D}\left[\phi(s,x)w_{R}(s,x)\right]_{s=0}^{T}dx-\int_{0}^{T}\int_{D}\phi(s,x)\partial_{s}w_{R}(s,x)dxds-\int_{0}^{T}\int_{D}\phi(s,x)L_{R}^{*}(s)w_{R}(s,x)dxds\\
 & =\int_{D}\phi(T,x)w_{R}(T,x)dx-\int_{D}\phi_{0}(x)w_{R}(0,x)dx.
\end{align*}
That is, in terms of $w$ and $\Phi$, 
\begin{equation}
\int_{D}\Phi(0,T)\phi_{0}(x)\cdot w(0,x)dx=\int_{D}\phi_{0}(x)w(T,x)dx.\label{eq:ByPartsTerm_w0}
\end{equation}
Impose the initial condition, 
\begin{equation}
w(0,\cdot)=w_{T}=\mathcal{A}\varphi-\varphi.\label{eq:InitialCdn_w0}
\end{equation}
Then it follows from (\ref{eq:DirectionalDerivOfF}), (\ref{eq:ByPartsTerm_w0})
and (\ref{eq:InitialCdn_w0}) that 
\begin{align}
DF(\varphi)(\phi_{0})= & \int_{D}w_{0}(x)\Phi(0,T)\phi_{0}(x)dx-\int_{D}w_{0}(x)\phi_{0}(x)dx=\int_{D}\left[w(T,x)-w_{0}(x)\right]\phi_{0}(x)dx.\label{eq:DirectionalDeriv_adjointed}
\end{align}
Since $\phi_{0}$ was arbitrary, by (\ref{eq:OperatorW}), we attain
(\ref{eq:Gradient_Functional_Expression}).
\end{proof}
We note that while Lemma \ref{lem:ConvexLemma} yields a unique minimiser,
it was not immediate whether $F$ was minimised to zero. In the following
theorem, we show indeed that the unique minimiser of $F$ indeed minimises
$F$ to zero. 
\begin{thm}
\label{thm:ConvexThm} Let Condition A2 hold and $d\leq3$, $f\in C^{\gamma}(0,T;H^{2}(D))$
for some $\gamma\in(0,1)$ and $f\geq0$. Then $v_{0}\in\mathscr{C}(D)$
obtained in Theorem \ref{thm:FunctionalUniqueSoln} is the unique
function in $L^{2}(D)$ satisfying (\ref{eq:ImposedPeriodForParabolicPDE}).
Moreover there exists a unique $v\in C(\bar{D}_{T})\cap C^{1,2}(D_{T})$
satisfying (\ref{eq:ReversedPDE_Periodic}). 
\end{thm}

\begin{proof}
By Theorem \ref{thm:FunctionalUniqueSoln}, the functional $F$ has
a unique minimiser $v_{0}\in\mathscr{C}(D)$. By Lemma \ref{lem:ConvexLemma}
and (\ref{eq:DirectionalDeriv_adjointed}), it follows that 
\[
DF(v_{0})(\phi_{0})=\int_{D}(w(T,x)-w_{0}(x))\phi_{0}(x)dx=0,\quad\phi_{0}\in L^{2}(D).
\]
By the fundamental lemma of calculus of variations, we have by (\ref{eq:OperatorW})
\[
0\equiv w(T,\cdot)-w_{0}(\cdot)=W(0,T)w_{0}-w_{0},
\]
i.e. $w_{0}$ is a fixed point of $W(0,T)$. Clearly $w_{0}\equiv0$
is a fixed point of $W(0,T)$. Let $w_{0}\in H^{2}(D)$ be another
fixed point solution to $W(0,T)$ and define $w(s,\cdot)$ by (\ref{eq:OperatorW}).
With $d\leq3$, by (\ref{eq:RegularSoln}), it follows that $w\in C(\bar{D}_{T})\cap C^{1+\frac{\theta}{2},2+\theta}((0,T]\times\bar{D})$.
In fact, since $D$ is bounded, $w(s,\cdot)\in L^{\infty}(D)$ for
$s\in[0,T]$. Note that $p_{D}$ of (\ref{eq:FokkerPlanck_AbsorbingBoundaries})
is a fundamental solution of (\ref{eq:PDE_for_w}), hence since $w_{0}$
is a fixed point of $W(0,T)$, it follows that 
\[
w_{0}(x)=\int_{D}p_{D}(0,T,x,y)w_{0}(y)dy,\quad x\in D.
\]
Due to the absorbing boundaries of (\ref{eq:PDE_for_w}), note that
for any $\Gamma\in\mathcal{B}(\mathbb{R}^{d})$, 
\begin{align*}
P_{D}(s,t,x,\Gamma) & :=\int_{D\cap\Gamma}p_{D}(s,t,x,y)dy\\
 & =\mathbb{P}(\{X_{t}^{s,x}\in D\cap\Gamma\}\cap\bigcap_{r=s}^{t}\{X_{r}^{s,x}\in D\})\\
 & \leq\mathbb{P}(\{X_{t}^{s,x}\in D\cap\Gamma\})\\
 & =P(s,t,x,D).
\end{align*}
Hence with $\epsilon\in(0,1)$ from Lemma \ref{lem:AlmostSurelyFinite_Eta_and_Expectation},
it follows that 
\[
\lVert w_{0}\rVert_{\infty}\leq\lVert w_{0}\rVert_{\infty}\int_{D}p(0,T,x,y)dy\leq\epsilon\lVert w_{0}\rVert_{\infty}.
\]
Thereby deducing $0\in H^{2}(D)$ is the only fixed point of $W(0,T)$.
Therefore, from (\ref{eq:InitialCdn_w0}), $v_{0}\in\mathscr{C}(D)$
is the unique minimiser such that $\mathcal{A}v_{0}=v_{0}$ and $F(v_{0})=0$.
It follows then from (\ref{eq:RegularSoln}) that 
\[
v(s,x):=\Phi(0,s)v_{0}(x)+\int_{0}^{s}\Phi(r,s)f(r)dr\in C(\bar{D}_{T})\cap C^{1+\frac{\theta}{2},2+\theta}((0,T]\times\bar{D})
\]
satisfies (\ref{eq:ReversedPDE_Periodic}). We show that $v_{0}$
is the unique fixed point of $\mathcal{A}$ in the entire $L^{2}(D)$.
Indeed suppose there exists another solution $\tilde{v}_{0}\in L^{2}(D)\backslash\mathscr{C}(D)$
such that $\tilde{v}_{0}=\mathcal{A}\tilde{v}_{0}$. By (\ref{eq:RegularSoln}),
$\mathcal{A}\tilde{v}_{0}\in C^{2+\theta}(\bar{D})\subset H^{2}(\bar{D})$
and satisfies the boundary conditions i.e. $\tilde{v}_{0}\in H_{0}^{2}(D)$.
Since $H^{2}(D)\ni f\geq0$, it follows from (\ref{eq:v0_is_nonnegative})
that $\tilde{v}_{0}\geq0$ i.e. $\tilde{v}_{0}\in\mathscr{C}(D)$.
Since uniqueness already holds in $\mathscr{C}(D)$, we conclude the
uniqueness of $v_{0}$ satisfying (\ref{eq:ImposedPeriodForParabolicPDE})
extends to $L^{2}(D)$.
\end{proof}
We summarise the PDE results in the context of expected duration $\bar{\tau}$
of SDEs. It will be clear that the SDE coefficients assumptions sufficiently
implies Conditions A2 required for Theorem \ref{thm:ConvexThm}. 
\begin{thm}
Let $d\le3$ and $D\subset\mathbb{R}^{d}$ be a non-empty open bounded
set. Assume that the $T$-periodic SDE (\ref{eq:GeneralSDE_NonAuton})
satisfies (\ref{eq:RegularityCondition}), (\ref{eq:LocallySmoothDrift}),
(\ref{eq:LocallySmoothSigma}) and (\ref{eq:Sigma_and_inverse_bounded}).
Then the expected duration $\bar{\tau}$ exists and is unique in $C(0,T;L^{2}(D))$.
In fact, $\bar{\tau}\in C(\bar{D}_{T})\cap C^{1,2}(D_{T})$ is non-negative
and non-trivial. 
\end{thm}

\begin{proof}
Specific to expected duration, we let $f(s,x)=I_{D}(x)$, the indicator
function on $D$ for all $s\in[0,T]$. Since $D$ is bounded, obviously
$f\in C^{\gamma}(0,T;H^{2}(D))$ is non-negative. Then by Theorem
\ref{thm:ConvexThm}, there exists a unique non-trivial non-negative
solution $v\in C(\bar{D}_{T})\cap C^{1,2}(D_{T})$ satisfying (\ref{eq:ReversedPDE_Periodic}).
Then by time-reversal, 
\[
\bar{\tau}(s,\cdot):=v(T-s,\cdot)=\Phi(0,T-s)v_{0}+\int_{0}^{T-s}\Phi(r,T-s)I_{D}dr,\quad s\in[0,T],
\]
satisfies (\ref{eq:ExpectedDurationPDE_Periodic}). 
\end{proof}

\section{Applications}

\subsection{Numerical Considerations}

As discussed in the introduction, expected exit times have a range
of applications including modelling the occurrence of certain events.
Depending on the context, the problem is typically phrased as the
stochastic process hitting a barrier or a threshold. While many physical
problems have naturally bounded domains, some applications have unbounded
domains. For example, $D=(0,\infty)$ is a typical unbounded domain
for species population or a wealth process, and exit from $D$ implies
extinction and bankruptcy respectively.

However, unbounded domain brings various technical difficulties for
the expected duration PDE. Particularly from a computational viewpoint,
any numerical PDE scheme requires a bounded domain. In the following
remark, we show that the recurrency condition (\ref{eq:NegativeOutsideB})
below is sufficient to approximate the expectation duration from an
unbounded domain to a bounded domain.
\begin{rem}
\label{rem:UnboundedToBounded}Suppose that there exists a radius
$r_{*}>0$ and $\epsilon>0$ such that the coefficients of SDE (\ref{eq:GeneralSDE_NonAuton})
satisfies 
\begin{equation}
2\langle b(t,x),x\rangle+\lVert\sigma\rVert_{2}^{2}\leq-\epsilon\quad\text{on }\mathbb{R}^{+}\times B_{r_{*}}^{c}.\label{eq:NegativeOutsideB}
\end{equation}
Note that if $b$ is continuous, then there exists a constant $M\geq-\epsilon$
such that $2\langle b(t,x),x\rangle+\lVert\sigma\rVert_{2}\leq M$.
Let $D$ be an unbounded domain that is bounded in at least one direction
hence $r_{D}:=\inf_{y\in\partial D}\lVert y\rVert$ is finite. We
suppose initial condition $x\in B_{r_{I}}\cap D$ for some fixed $r_{I}\geq0$.

For any fixed $R_{*}>\max\{r_{*},r_{D},r_{I}\}$, define $\tilde{D}:=D\cap B_{R_{*}}$.
Note that $\tilde{D}$ is an open bounded domain with boundary $\partial\tilde{D}=\partial\tilde{D}_{1}\cup\partial\tilde{D}_{2}$,
where $\partial\tilde{D}_{1}:=\partial D\cap B_{R_{*}}$ and $\partial\tilde{D}_{2}:=D\cap\partial B_{R_{*}}$
are the subset of original boundary and ``artificial'' boundary
to ``close up'' the original boundaries respectively. Observe that
$\lVert x\rVert\in(r_{D},R_{*})$ for $x\in\partial\tilde{D}_{1}$
and $\lVert x\rVert=R_{*}$ for $x\in\partial\tilde{D}_{2}$.

For $\eta_{\tilde{D}}$ as defined by (\ref{eq:FirstExitTime}) for
the domain $\tilde{D}$, by Itô's formula, we have

\[
\lVert X_{t\wedge\eta_{\tilde{D}}}\rVert^{2}=\lVert x\rVert^{2}+\int_{s}^{t\wedge\eta_{\tilde{D}}}2\langle b(r,X_{r}),X_{r}\rangle+\lVert\sigma\rVert_{2}^{2}dr+\int_{s}^{t\wedge\eta_{\tilde{D}}}\langle X_{r},\sigma dW_{r}\rangle.
\]
Under the assumption (\ref{eq:NegativeOutsideB}), Corollary 3.2 of
\cite{Hasminskii} implies that $\mathbb{E}^{s,x}(\eta_{\tilde{D}}-s)\leq\frac{\lVert x\rVert^{2}}{\epsilon}$.
Since $X_{r\wedge\eta_{\tilde{D}}}$ is bounded for $r\in[s,t]$,
it follows by taking expectations that 
\[
\mathbb{E}^{s,x}\lVert X_{t\wedge\eta_{\tilde{D}}}\rVert^{2}\leq\lVert x\rVert^{2}+M\mathbb{E}^{s,x}[\eta_{\tilde{D}}-s]\leq(1+\frac{1}{\epsilon})\lVert x\rVert^{2}.
\]
By Markov's inequality, it follows that $\mathbb{P}(\lVert X_{\eta_{\tilde{D}}}\rVert^{2}\geq R_{*}^{2})\leq R_{*}^{-2}(1+\frac{1}{\epsilon})\lVert x\rVert^{2}\leq R_{*}^{-2}(1+\frac{1}{\epsilon})r_{I}^{2}\rightarrow0$
as $R_{*}\rightarrow\infty$. This implies that for sufficiently large
$R_{*}$, the process exits $\tilde{D}$ via $\partial\tilde{D}_{1}$
rather than $\partial\tilde{D}_{2}$, thus 
\[
\bar{\tau}_{D}\rvert_{\tilde{D}}(s,\cdot)\simeq\bar{\tau}_{\tilde{D}}(s,\cdot),
\]
where $\bar{\tau}_{D}\rvert_{\tilde{D}}$ denotes $\bar{\tau}_{D}$
restricted to $\tilde{D}$. In practice, $R_{*}=2\max\{r_{*},r_{D},r_{I}\}$
is sufficient for weakly dissipative SDEs. We consider two examples
and assume for simplicity that $r_{I}=r_{*}$.
\end{rem}

It was shown in \cite{FZZ19} that the periodic Ornstein-Uhlenbeck
process possesses a geometric periodic measure \cite{CFHZ2016}, furthermore
it has a periodic mean reversion property akin to its classical counterpart
with (constant) mean reversion. In applications, these properties
are desirable for processes with underlying periodicity or seasonality.
Indeed electricity prices in economics \cite{PeriodicOU1,PeriodicOU2}
and daily temperature \cite{PeriodicOU3} were modelled by periodic
Ornstein-Uhlenbeck processes. In \cite{PeriodicOU_Neuroscience},
the authors performed statistical inference of biological neurons
modelled by Ornstein-Uhlenbeck processes with periodic forcing. In
such models, one may be interested in the expected time in which a
threshold is reached. For model parameter estimation, we refer readers
to \cite{PeriodicOU_drift_estimation}. 
\begin{example}
\label{exa:WeaklyDissip_OU} Consider the periodically forced multi-dimensional
Ornstein-Uhlenbeck process 
\[
dX_{t}=\left(S(t)-AX_{t}\right)dt+\sigma dW_{t},
\]
where $S\in C(\mathbb{R}^{+},\mathbb{R}^{d})$ is $T$-periodic and
$\sigma,A\in\mathbb{R}^{d\times d}$ with $A$ positive definite i.e.
there exists a constant $\alpha>0$ such that $\langle Ax,x\rangle\geq\alpha\lVert x\rVert^{2}$
for all $x\in\mathbb{R}^{d}$. Denote $\lVert S\lVert_{\infty}=\sup_{t\in[0,T]}\lVert S(t)\lVert$.
By Cauchy-Schwarz inequality and Young's inequality, it follows that
\begin{align*}
2\langle S(t)-Ax,x\rangle & \leq2\lVert S\lVert_{\infty}\lVert x\lVert-2\alpha\lVert x\rVert^{2}\leq\frac{\lVert S\lVert_{\infty}^{2}}{\alpha}-\alpha\lVert x\rVert^{2},
\end{align*}
i.e. weakly dissipative with coefficients $c=\frac{\lVert S\lVert_{\infty}^{2}}{\alpha}$
and $\lambda=\alpha$. Then 
\[
r_{*}^{2}=\frac{1}{\alpha}\left(\frac{\lVert S\lVert_{\infty}^{2}}{\alpha}+\lVert\sigma\rVert_{2}\right).
\]
Remark \ref{rem:UnboundedToBounded} suggests one can reasonably approximate
the unbounded domain $D=(0,\infty)^{d}$ by the bounded domain $\tilde{D}=(0,2r_{*})^{d}$.
We remark also, to our current knowledge, there is no closed-form
formula to solve (\ref{eq:ExpectedDurationPDE_Periodic}) even in
this simple example. As such, we provide the numerical solution to
(\ref{eq:ExpectedDurationPDE_Periodic}) in Section \ref{sec:NumericalSolutionOfPde}.
\end{example}

\begin{example}
\label{exa:DuffOsc_is_WeaklyDissip}Consider the stochastic overdamped
Duffing oscillator 
\begin{equation}
dX_{t}=\left(X_{t}-X_{t}^{3}+A\cos(\omega t)\right)dt+\sigma dW_{t},\label{eq:DuffingOsc}
\end{equation}
where $A,\omega\in\mathbb{R}$ and $\sigma\in\mathbb{R\backslash}\{0\}$
are (typically small) parameters. By elementary calculus, it is straightforward
to show (\ref{eq:DuffingOsc}) satisfies the weakly dissipative condition
for any fixed $\lambda\in(0,2)$ and $c=\frac{1}{2-\lambda}+2\lvert A\rvert+\frac{\lambda}{4}$
for small $A$. Thus 
\[
r_{*}^{2}=\frac{1}{\lambda}\left(\frac{1}{2-\lambda}+2\lvert A\rvert+\lVert\sigma\rVert_{2}\right)+\frac{1}{4}.
\]
For concreteness, suppose that $A=0.12$, $\omega=10^{-3}$, $\sigma=2.85$
and $\lambda=1$, then $r_{*}=\sqrt{1.57}=1.25$ (2 dp). Remark \ref{rem:UnboundedToBounded}
suggests that the process exiting $D=(-1,\infty)$ can be approximated
by the bounded domain $\tilde{D}=(-1,2r_{*})$.

Via Monte Carlo simulations, we numerically demonstrate Remark \ref{rem:UnboundedToBounded}
for (\ref{eq:DuffingOsc}) to estimate $\bar{\tau}_{\tilde{D}}(0,x)$
for different bounded domains $\tilde{D}$. We partition $\tilde{D}$
into $N_{x}^{\text{sde}}=100$ uniform initial conditions. For each
fixed initial condition $x\in\tilde{D}$, we employ Euler\textendash Maruyama
method with time intervals of $\Delta t=5\cdot10^{-3}$ to generate
1000 sample-paths of $X_{t}$ until it exits $\tilde{D}$. We record
and average the sample-path exit time to yield an estimate for $\bar{\tau}_{\tilde{D}}(0,x)$.
Figure \ref{DifferentDomainGraph} shows that the estimation of $\bar{\tau}_{\tilde{D}}$
are ``stable'' for bounded domain $\tilde{D}=(-1,2)$ and larger.
Where the differences between these curves can be explained by the
randomness of Monte Carlo simulations and sample size. On the other
hand, the estimation of $\bar{\tau}_{\tilde{D}}$ differs significantly
for $\tilde{D}=(-1,1.5)$. Physically, this is interpreted as the
artificial boundary $R_{*}=1.5$ set too low and many sample-paths
leave via this artificial boundary rather than via $-1$.

\begin{figure}[h]
\begin{centering}
\includegraphics[scale=0.55]{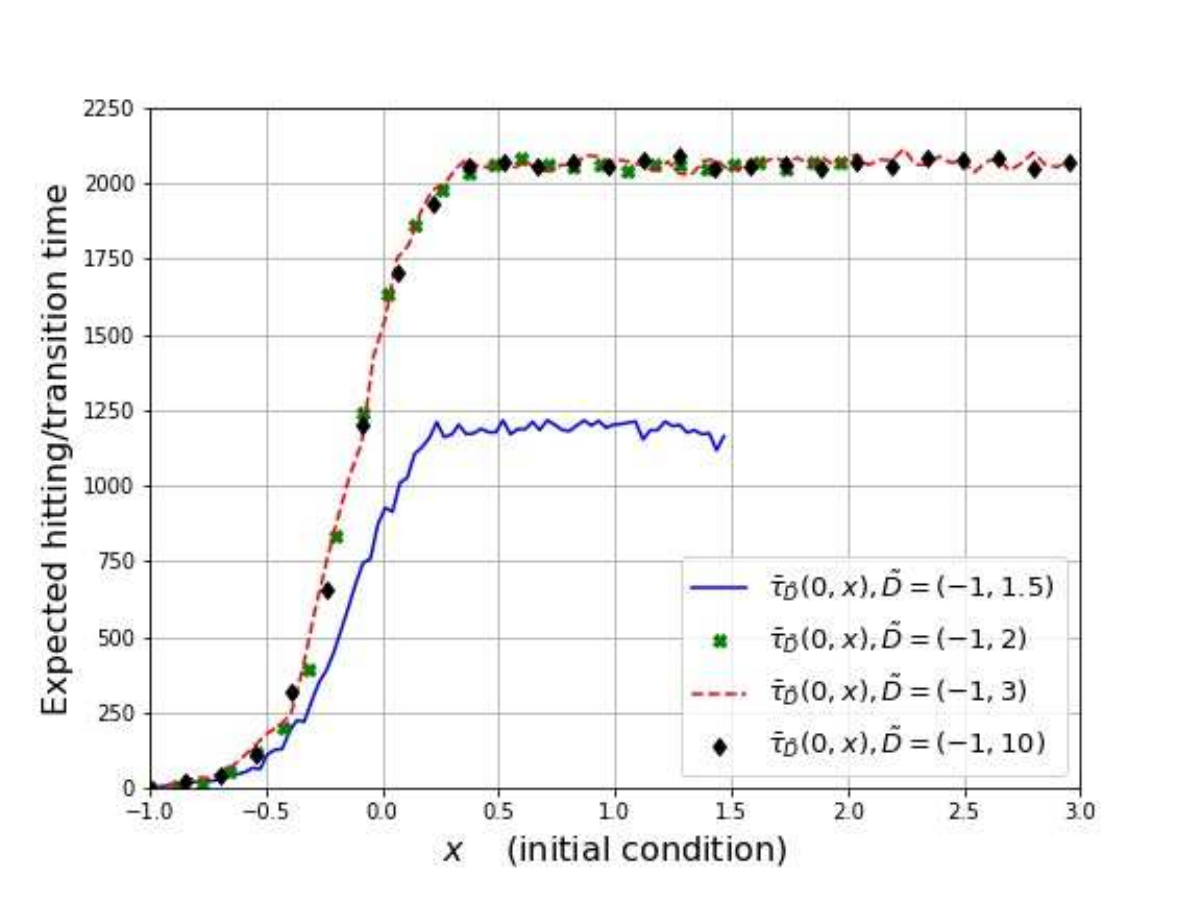} 
\par\end{centering}
\caption{Monte Carlo estimation of $\bar{\tau}_{\tilde{D}}(0,x)$ with different
$\tilde{D}$, plotted for $x\in(-1,3)\cap\tilde{D}$.}
\label{DifferentDomainGraph} 
\end{figure}

Finally, for subsequent analysis, while $\tilde{D}=(-1,2)$ is sufficient,
we will reduce $D$ to $\tilde{D}=(-1,3)$. We pick this larger domain
to accommodate when we use $\sigma=1$ where $r_{*}=1.58$ (2 dp).
\end{example}

\begin{rem}
\label{rem:shapeOfTau}We remark that in Figure \ref{DifferentDomainGraph}
for stochastic resonance problem such as (\ref{eq:DuffingOsc}), $\bar{\tau}(s,x)$
visually resembles a step-function is intuitively expected. We indeed
expect $\bar{\tau}(s,x)$ small when the process starts near the left
well because the process does not need to cross the potential well
to exit. On the other hand, $\bar{\tau}(s,x)$ is large when the process
starts in the right well. By the well-posedness of the PDE, $\bar{\tau}(s,x)$
is continuous and hence we expect a sharp increase between the two
wells. Therefore, for the stochastic resonance problem, we expect
$\bar{\tau}(s,x)$ to appear like a step-function.
\end{rem}

\subsection{Numerical Solution of Expected Duration PDE\label{sec:NumericalSolutionOfPde}}

In this section, we describe how one may numerically solve the PDE
(\ref{eq:ReversedPDE_Periodic}) using Theorem \ref{thm:ExistenceUniquenessResult_Coercive}
and Theorem \ref{thm:ConvexThm}. For simplicity of this exposition,
we do this in the one-dimensional case only. We remark that the presented
numerical procedure does not claim originality nor optimal efficiency
nor optimal accuracy as it is not the purpose of the current paper.
We state and prove the following lemma that would be useful in applying
Theorem \ref{thm:ConvexThm} numerically. It provides an analytical
step size in numerically computing the next iterate, given the initial
point and iteration direction.
\begin{lem}
\label{lem:GradientStepSize}The function $\Psi(\gamma):=F(v+\gamma\varphi)$
for $v,\phi\in L^{2}(D)$ is minimised when

\begin{equation}
\gamma=-\frac{\langle\mathcal{A}v-v,\Phi(0,T)-I)\phi\rangle_{L^{2}(D)}}{\lVert(\Phi(0,T)-I)\phi\rVert_{L^{2}(D)}^{2}}.\label{eq:gradientStepSize}
\end{equation}
\end{lem}

\begin{proof}
Using the identities (\ref{eq:DirectionalDerivOfF}) and $\mathcal{A}(v+\gamma\phi)=\mathcal{A}v+\gamma\Phi(0,T)\phi$
as well as the linearity property of $\Phi(0,T)$, we have that
\begin{align*}
\dot{\Psi}(\gamma) & =DF(v+\gamma\phi)(\phi)\\
 & =\int_{D}(\mathcal{A}(v+\gamma\phi)-(v+\gamma\phi))(\Phi(0,T)-I)\phi dx\\
 & =\int_{D}\left(\mathcal{A}v-v\right)(\Phi(0,T)-I)\phi)dx+\gamma\int_{D}(\Phi(0,T)-I)\phi)^{2}dx.
\end{align*}
It follows that $\dot{\Psi}(\gamma)=0$ when $\gamma$ is given by
(\ref{eq:gradientStepSize}). Note in fact, $F(v)=0$ is already at
a minimum when $\mathcal{A}v=v$ and hence $\gamma=0$ using (\ref{eq:gradientStepSize}).
\end{proof}
We now describe the numerical PDE scheme to solve (\ref{eq:ReversedPDE_Periodic}).
For given spatial domain $D=(d_{\min},d_{\max})$, we partition it
into $N_{x}^{\text{pde}}=500$ uniform interior nodes and the time
interval $[0,T]$ into $N_{t}^{\text{pde}}=500$ uniform points. To
evolve the parabolic IBVPs (\ref{eq:TimeReversedHomo_IBVP}) and (\ref{eq:PDE_for_w})
in time, we implement a backward Euler scheme with central differencing.
We note that in problems such as stochastic resonance where the period
$T$ is large, backward Euler method is preferred over Crank-Nicholson
method to enable larger time steps and less prone to spurious oscillations.

To apply Banach Fixed Point numerically of Theorem \ref{thm:ExistenceUniquenessResult_Coercive},
we seek an finite representation of $\mathcal{A}$ given by (\ref{eq:OperatorA}).
We first discretise the continuous solution $v(t,x)$ to the IBVP
(\ref{eq:Inhomo_Reversed_IBVP}) with 
\[
v_{i}^{n}:=v(t^{n},x_{i}),\quad0\leq n\leq N_{t}^{\text{pde}},0\leq i\leq N_{x}^{\text{pde}},
\]
where $x_{i}=d_{\min}+i\Delta x,\quad i=0,..,N_{x}^{\text{pde}}+1$,
$\Delta x=\frac{d_{\max}-d_{\min}}{N_{x}^{\text{pde}}+1}$, $t^{n}=n\Delta t,\quad n=0,..,N_{t}^{\text{pde}}$
and $\Delta t=\frac{T}{N_{t}^{\text{pde}}}$. Using backward time-differencing
and central spatial differencing, it is straightforward to attain
the following parabolic PDE discretisation for (\ref{eq:TimeReversedHomo_IBVP})

\begin{align*}
\left(1+a_{i}^{n}\frac{\Delta t}{(\Delta x)^{2}}\right)u_{i}^{n+1} & =u_{i}^{n}+\left(\frac{a_{i}^{n}}{2}\frac{\Delta t}{(\Delta x)^{2}}+\frac{b_{i}^{n}}{2}\frac{\Delta t}{\Delta x}\right)u_{i+1}^{n+1}+\left(\frac{a_{i}^{n}}{2}\frac{\Delta t}{(\Delta x)^{2}}-\frac{b_{i}^{n}}{2}\frac{\Delta t}{\Delta x}\right)u_{i-1}^{n+1},
\end{align*}
for $1\leq i\leq N_{x}^{\text{pde}},0\leq n\leq N_{t}^{\text{pde}}$
and $a_{i}^{n}=a(T-t^{n},x_{i})$ and $b_{i}^{n}=b(T-t^{n},x_{i})$
are the time-reversed coefficients evaluated at the grid points. We
then define the tridiagonal matrix $\hat{\Phi}^{-1}(n)\in\mathbb{R}^{N_{x}^{\text{pde}}\times N_{x}^{\text{pde}}}$
for $1\leq n\le N_{t}^{\text{pde}}$ by 
\[
\hat{\Phi}^{-1}(n):=\begin{cases}
i\text{'th diagonal: \ensuremath{\quad1+a_{i}^{n}\frac{\Delta t}{(\Delta x)^{2}},}}\\
i\text{'th superdiagonal: }\ensuremath{\quad-\frac{a_{i+1}^{n}}{2}\frac{\Delta t}{(\Delta x)^{2}}-\frac{b_{i+1}^{n}}{2}\frac{\Delta t}{\Delta x}}\ensuremath{,}\\
i\text{'th subdiagonal: \ensuremath{\quad}}\frac{b_{i-1}^{n}}{2}\frac{\Delta t}{\Delta x}-\frac{a_{i-1}^{n}}{2}\frac{\Delta t}{(\Delta x)^{2}}.
\end{cases}
\]

Hence, for a given $f\in C^{\gamma}(0,T;L^{p}(D))$, define $\hat{f}^{n}=(f(T-t^{n},x_{1}),...,f(T-t^{n},x_{N_{x}^{\text{pde}}}))^{T}$,
together with homogeneous Dirichlet boundary conditions, we have the
following linear system for (\ref{eq:Inhomo_Reversed_IBVP})
\begin{equation}
\hat{\Phi}^{-1}(n)\hat{v}^{n+1}=\hat{v}^{n}+\Delta t\cdot\hat{f}^{n},\quad n=1,...,N_{t}^{\text{pde}}.\label{eq:tridiagonalEqn}
\end{equation}

Equation (\ref{eq:tridiagonalEqn}) implies that $\hat{v}^{n+1}=\hat{\Phi}(n)(v+\Delta t\cdot\hat{f}^{n})$
and so $\hat{\Phi}(n)$ is the finite dimensional analogue of $\Phi(t^{n},t^{n+1})$
given by (\ref{eq:EvolutionOperator}). Since $\hat{\Phi}^{-1}(n)$
is tridiagonal, equation (\ref{eq:tridiagonalEqn}) can be efficiently
solved by utilising lower decomposition or iterative methods which
are generally more efficient than explicitly inverting the matrix.
By iteratively solving, one evolves the IBVP (\ref{eq:TimeReversedHomo_IBVP})
from initial time $t=0$ to time $t=T$ for some given initial conditions.
Define $\hat{\mathcal{A}}^{-1}:=\hat{\Phi}^{-1}(1)\hat{\Phi}^{-1}(2)\cdot\cdot\cdot\hat{\Phi}^{-1}(N_{t}^{\text{pde}})$,
then it follows that

\begin{equation}
\hat{v}^{N_{t}^{\text{pde}}}=\hat{\mathcal{A}}\hat{v}^{0},\quad\hat{v}^{0}\in\mathbb{R}^{N_{x}^{\text{pde}}},\label{eq:EvolveParabolic}
\end{equation}
numerically solves (\ref{eq:Inhomo_Reversed_IBVP}) from initial time
$t=0$ to time $t=T$. Note further that $\hat{\mathcal{A}}$ is the
finite counterpart of (\ref{eq:OperatorA}), the period solution map
of (\ref{eq:Inhomo_Reversed_IBVP}). Given $\hat{A}\hat{v}$, we define
the finite dimensional cost functional 
\[
\hat{F}(\hat{v}):=\frac{\sqrt{\Delta x}}{2}\lVert\hat{\mathcal{A}}\hat{v}-\hat{v}\rVert_{\mathbb{R}^{N_{x}^{\text{pde}}}}^{2},
\]
where $\lVert\cdot\rVert_{\mathbb{R}^{N_{x}^{\text{pde}}}}$ is the
standard $\mathbb{R}^{N_{x}^{\text{pde}}}$ Euclidean norm. Thus,
to apply Banach Fixed Point iteration of Theorem \ref{thm:ExistenceUniquenessResult_Coercive},
one can arbitrary choose initial vector $\hat{v}_{0}\in\mathbb{R}^{N_{x}^{\text{pde}}}$
and iterate $\hat{v}_{n+1}=\hat{\mathcal{A}}\hat{v}_{n}$ until $\hat{F}(\hat{v}_{n})\leq10^{-6}$.
While $\hat{v}_{0}$ can be chosen arbitrarily, since $\bar{\tau}(s,x)\geq0$,
see (\ref{eq:v0_is_nonnegative}), it is recommended to choose $\hat{v}_{0}\geq0$
also. Particular to stochastic resonance problems, by Remark \ref{rem:shapeOfTau},
it is recommended to pick $\hat{v}_{0}$ by a step-like function such
as a scaled sigmoid function $\frac{1}{\sigma^{2}(1+e^{-x})}$. The
scaling is due to our anticipation that the expected duration time
is inversely proportional to the diffusion coefficient.

Given that the $F$ is strictly convex and we have the gradient (\ref{eq:Gradient_Functional_Expression}),
we can numerically solve Theorem \ref{thm:ConvexThm} by a gradient
descent method scheme. For simplicity we assume $a=\text{const}$
and again apply backward time differencing and central spatial differencing,
we have the discretisation for (\ref{eq:PDE_for_w})

\begin{align*}
 & \left(1+a(t^{n},x_{i})\frac{\Delta t}{(\Delta x)^{2}}+\partial_{x}b(t^{n},x_{i})\Delta t\right)u_{i}^{n+1}\\
 & =u_{i}^{n}+\left(\frac{a(t^{n},x_{i})}{2}\frac{\Delta t}{(\Delta x)^{2}}-\frac{b(t^{n},x_{i})}{2}\frac{\Delta t}{\Delta x}\right)u_{i+1}^{n+1}+\left(\frac{a(t^{n},x_{i})}{2}\frac{\Delta t}{(\Delta x)^{2}}+\frac{b(t^{n},x_{i})}{2}\frac{\Delta t}{\Delta x}\right)u_{i-1}^{n+1}.
\end{align*}
Note here, the coefficients are not time-reversed. We similarly define
the following tridiagonal matrices $W^{-1}(n)\in\mathbb{R}^{N_{x}^{\text{pde}}\times N_{x}^{\text{pde}}}$for
$n=0,...,N_{t}^{\text{pde}}$ by
\[
\hat{W}^{-1}(n):\begin{cases}
i\text{'th diagonal: \ensuremath{\quad1+a(t^{n},x_{i})\frac{\Delta t}{(\Delta x)^{2}}+\partial_{x}b(t^{n},x_{i})\Delta t,}}\\
i\text{'th superdiagonal:}\ensuremath{\quad-\frac{a(t^{n},x_{i+1})}{2}\frac{\Delta t}{(\Delta x)^{2}}}\ensuremath{,}\\
i\text{'th subdiagonal:\ensuremath{\quad}}-b(t^{n},x_{i-1})\frac{\Delta t}{\Delta x}-\frac{a(t^{n},x_{i-1})}{2}\frac{\Delta t}{(\Delta x)^{2}}.
\end{cases}
\]
We analogously define $\hat{W}^{-1}(n,m):=\hat{W}^{-1}(n)\hat{\Phi}_{W}^{-1}(n+1)\cdot\cdot\cdot\hat{W}^{-1}(m)$,
where $0\leq n<m<N_{t}^{\text{pde}}$ hence $\hat{W}(1,N_{t}^{\text{pde}})$
is the finite counterpart to $W(0,T)$ defined in (\ref{eq:OperatorW}).
As with (\ref{eq:EvolveParabolic}), we evolve (\ref{eq:PDE_for_w})
by solving the triangular system of equations. Unlike numerically
solving Theorem \ref{thm:ExistenceUniquenessResult_Coercive}, the
initial conditions in (\ref{eq:PDE_for_w}) is not chosen arbitrarily.
Instead given $\hat{v}_{n}$ , the vector
\[
\hat{w}_{n}:=\mathcal{\hat{A}}\hat{v}_{n}-\hat{v}_{n}\in\mathbb{R}^{N_{x}^{\text{pde}}}
\]
serves as the $n$'th initial condition to (\ref{eq:PDE_for_w}).
Therefore, by (\ref{eq:Gradient_Functional_Expression}), the finite
dimensional gradient is given by
\[
\hat{\frac{\delta F}{\delta v_{n}}}=\hat{W}(1,N_{t}^{\text{pde}})\hat{w}_{n}-\hat{w}_{n}\in\mathbb{R}^{N_{x}^{\text{pde}}}.
\]
Then one can compute the gradient descent iterates by
\begin{equation}
\hat{v}_{n+1}=\hat{v}_{n}-\gamma_{n}\hat{\frac{\delta F}{\delta v_{n}}},\label{eq:gradientIterates}
\end{equation}
that is we pick the next iterate in the direction $\phi=-\hat{\frac{\delta F}{\delta v_{n}}}$
with some step size $\gamma_{n}$. By Lemma \ref{lem:GradientStepSize},
we can pick $\gamma_{n}$ optimally with the formula
\[
\gamma_{n}=\frac{\langle\hat{w}_{n},(\hat{\Phi}(1,N_{t}^{\text{pde}})-I)\hat{\frac{\delta F}{\delta v_{n}}}\rangle_{\mathbb{R}^{N_{x}^{\text{pde}}}}}{\lVert(\hat{\Phi}(1,N_{t}^{\text{pde}})-I)\hat{\frac{\delta F}{\delta v_{n}}}\rVert_{\mathbb{R}^{N_{x}^{\text{pde}}}}^{2}},
\]
as the finite dimensional counterpart of (\ref{eq:gradientStepSize})
replacing $L^{2}(D)$ inner-product with $\mathbb{R}^{N_{x}^{\text{pde}}}$
Euclidean inner-product. To compute $\gamma_{n}$, one just need to
compute $(\hat{\Phi}(1,N_{t}^{\text{pde}})\hat{\frac{\delta F}{\delta v_{n}}}$
by again solving (\ref{eq:tridiagonalEqn}) where $\hat{f}^{n}\equiv0$
and computing the inner-product. Thus to perform the gradient descent
procedure, we again choose an arbitrary initial $\hat{v}_{0}\in\mathbb{R}^{N_{x}^{\text{pde}}}$
and iterate (\ref{eq:gradientIterates}) until $\hat{F}(v_{n})\leq10^{-6}$.

Using the procedure described above, we end this section numerically
solve (\ref{eq:ReversedPDE_Periodic}) for the expected duration for
the Duffing Oscillator (\ref{eq:DuffingOsc}). Following Example \ref{exa:DuffOsc_is_WeaklyDissip},
we choose the same parameters $A=0.12$, $\omega=10^{-3}$ and $\sigma=0.285$.
The same parameters was considered in \cite{CherubiniStochRes}. As
Example \ref{exa:DuffOsc_is_WeaklyDissip} and Figure \ref{DifferentDomainGraph}
demonstrated, we reduce the unbounded domain to the bounded domain
$\tilde{D}=(-1,3)$. We then estimate the $s=0$ cross-section $\bar{\tau}_{0.285}(0,\cdot)$
by three approaches. In this demonstration, we let $\bar{\tau}_{0.285}^{\text{sde}}$,
$\bar{\tau}_{0.285}^{\text{bfp}}$ and $\bar{\tau}_{0.285}^{\text{grad}}$
to respectively represent the Monte Carlo simulation from Example
\ref{exa:DuffOsc_is_WeaklyDissip}, Banach fixed point iteration (Theorem
\ref{thm:ExistenceUniquenessResult_Coercive}) and gradient descent
iteration via convex optimisation (Theorem \ref{thm:ConvexThm}).
Figure \ref{SDE_vs_PDE_plot} shows these approximations.

\begin{figure}[h]
\centering{}\includegraphics[scale=0.75]{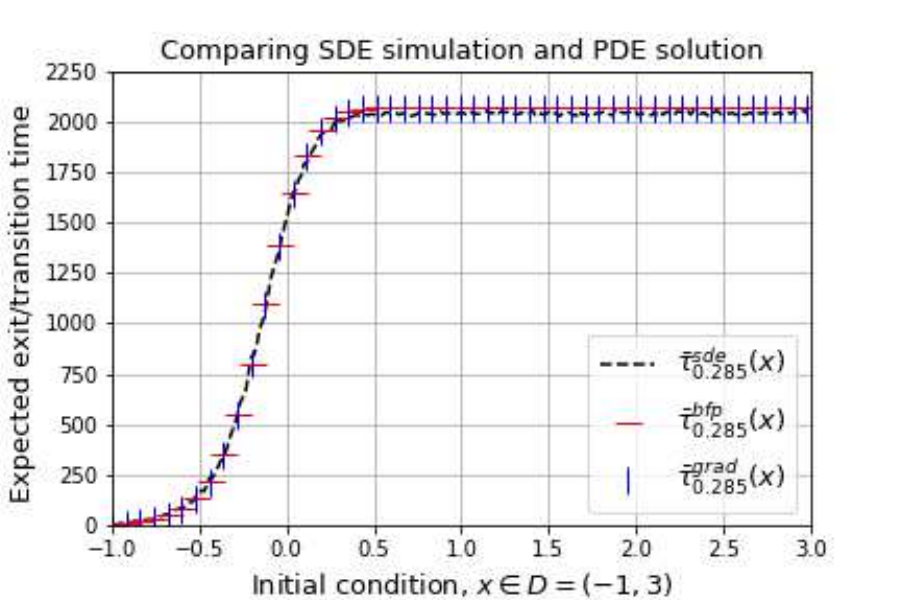}\caption{Numerical approximation of the expected transition time $\bar{\tau}_{0.285}(x)$
by $\bar{\tau}_{0.285}^{\text{sde}}(x)$, $\bar{\tau}_{0.285}^{\text{bfp}}(x)$
and $\bar{\tau}_{0.285}^{\text{grad}}(x)$ to SDE (\ref{eq:DuffingOsc})
with parameters $A=0.12,$ $\omega=0.001$, $\sigma=0.285$, $s=0$,
$D=(-1,3)$.}
\label{SDE_vs_PDE_plot}
\end{figure}
For stochastic simulation of Figure \ref{SDE_vs_PDE_plot}, $\bar{\tau}_{0.285}^{\text{sde}}$
from Example \ref{exa:DuffOsc_is_WeaklyDissip} for the domain $D=(-1,3)$
was reused. Figure \ref{SDE_vs_PDE_plot} shows that $\bar{\tau}_{0.285}^{\text{bfp}}$
and $\bar{\tau}_{0.285}^{\text{grad}}$ closely approximate each other
very well and in turn both visually approximate $\bar{\tau}_{0.285}^{\text{sde}}$
well, particularly for initial conditions starting in the right well.
In the absence of an analytic formulae of $\bar{\tau}_{0.285}$, we
take $\bar{\tau}_{0.285}^{\text{sde}}$ to be the ``true'' solution.
Hence we estimated the relative errors $\frac{\lVert\bar{\tau}_{0.285}^{\text{sde}}-\bar{\tau}_{0.285}^{\text{bfp}}\rVert_{L^{2}(\tilde{D})}}{\lVert\bar{\tau}_{0.285}^{\text{sde}}\rVert_{L^{2}(\tilde{D})}}=1.65\%$
(2 dp) and $\frac{\lVert\bar{\tau}_{0.285}^{\text{sde}}-\bar{\tau}_{0.285}^{\text{grad}}\rVert_{L^{2}(\tilde{D})}}{\lVert\bar{\tau}_{0.285}^{\text{sde}}\rVert_{L^{2}(\tilde{D})}}=1.74\%$
(2 dp). The small relative error validates approximating expected
transition time $\bar{\tau}_{0.285}$ by solving PDE (\ref{eq:ReversedPDE_Periodic})
numerically by either $\bar{\tau}_{0.285}^{\text{bfp}}$ or $\bar{\tau}_{0.285}^{\text{grad}}$
for the Duffing Oscillator. It is anticipated that the (relative)
errors can be explained by the particular numerical scheme and parameters
used in the PDE and SDE methods. It may be particularly remarkable
that the Banach fixed point iterates converges because it is not immediate
whether the associated bilinear form is coercive.

Figure \ref{SDE_vs_PDE_plot} is a cross-section of $\bar{\tau}_{\text{0.285}}(s,x)$
with $s=0$. With the knowledge that the PDE solution approximates
the SDE solution well, we produce Figure \ref{DuffingSurfacePlot},
a surface plot of $\bar{\tau}(s,x)$ to see how changing both initial
time and initial state changes the expected duration time. As expected,
the surface is indeed periodic in $s$.

\begin{figure}[h]
\centering{}\includegraphics[scale=0.55]{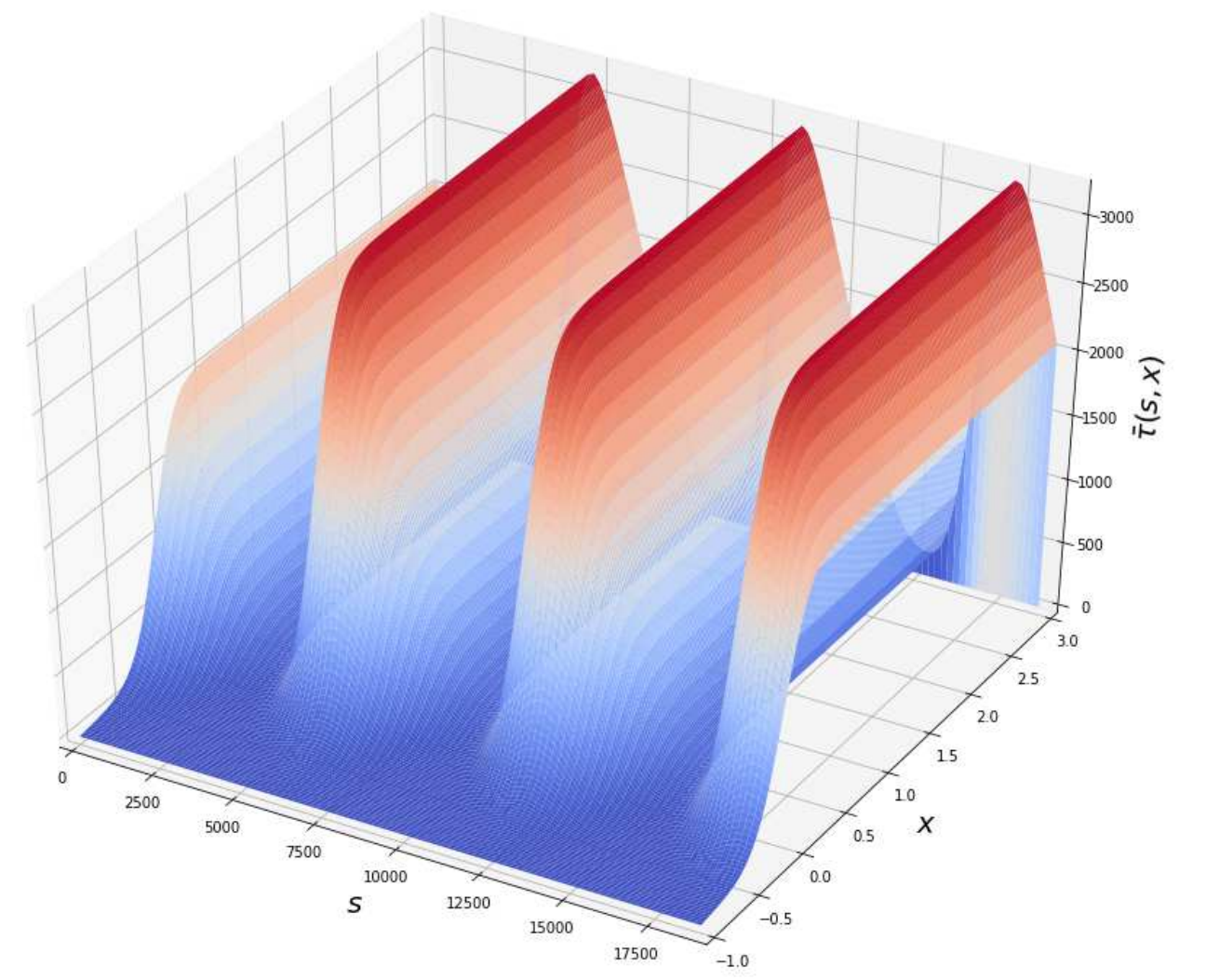}\caption{Numerical approximation of the expected duration $\bar{\tau}(s,x)$
of SDE (\ref{eq:DuffingOsc}) with parameters $A=0.12,$ $\omega=0.001$,
$\sigma=0.285$, $s=0$, $D=(-1,3)$ and $s\in[0,3T],$where $T=\frac{2\pi}{\omega}$
is the period.}
\label{DuffingSurfacePlot}
\end{figure}

Similarly, Figure \ref{SurfacePlot} shows the surface plot for Example
\ref{exa:WeaklyDissip_OU}.

\begin{figure}[h]
\centering{}\includegraphics[scale=0.55]{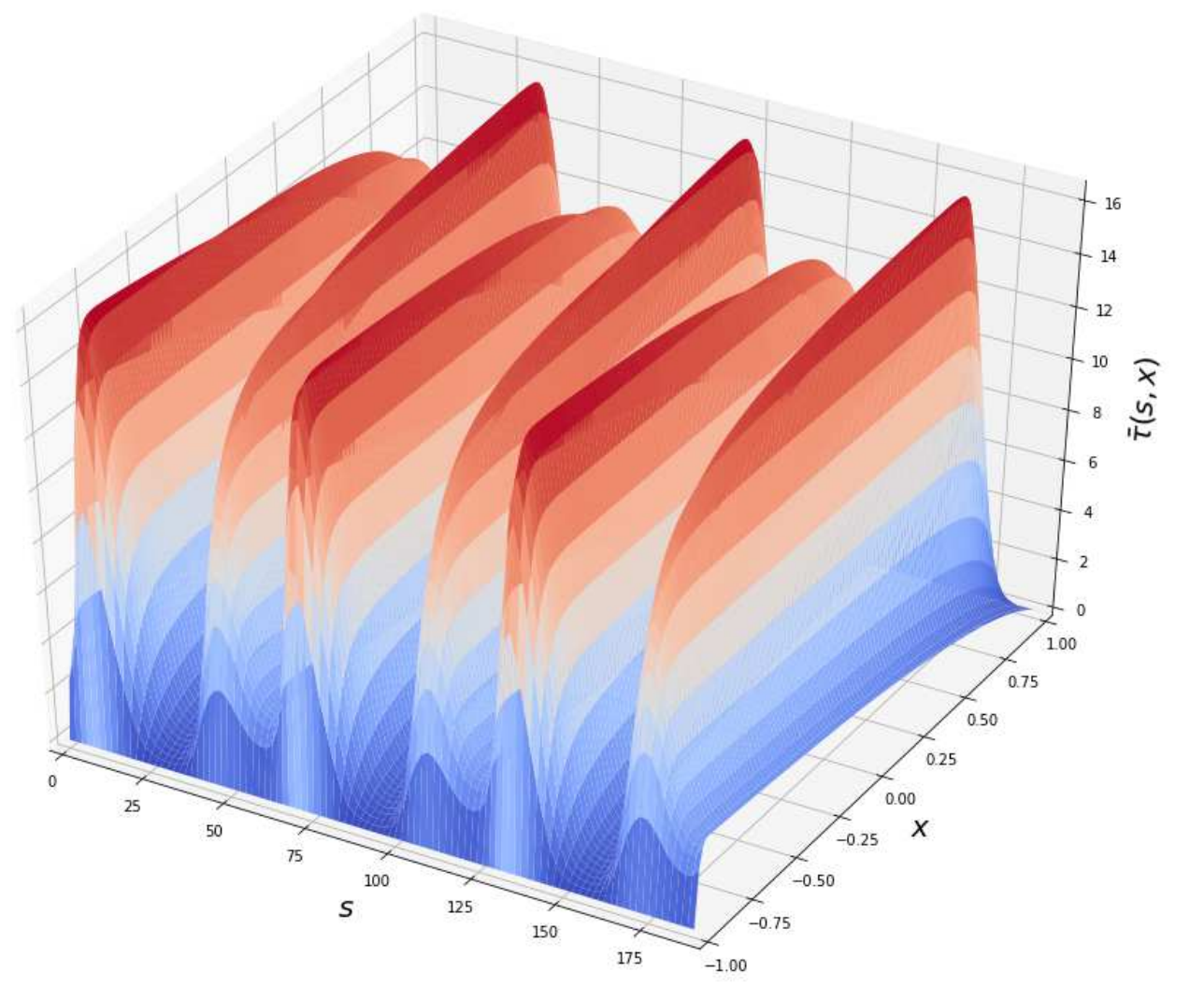}\caption{Numerical approximation of the expected duration $\bar{\tau}(s,x)$
of SDE $dX_{t}=\left(\cos(0.1t)-X_{t}\right)dt+0.285\cdot dW_{t}$
with the domain $D=(-1,1)$ and $s\in[0,3T],$where $T=\frac{2\pi}{\omega}$
is the period.}
\label{SurfacePlot}
\end{figure}

\subsection{Stochastic Resonance\label{subsec:Stochastic-Resonance}}

We now apply the results of this paper to study the physical phenomena
of stochastic resonance. In the introduction, we discussed the modelling
of stochastic resonance by a periodically-forced double-well potential
SDEs and the interest in the expected duration time between the two
wells. We refer to this time as expected transition time to align
with the physical interpretation of the problem. In \cite{FZZ19},
it was shown that time-periodic weakly dissipative SDEs, which includes
double-well potential SDEs, possesses a unique geometric periodic
measure. The existence and uniqueness of geometric periodic measure
of (\ref{eq:DuffingOsc}) implies that transitions between the metastable
states do occur as well as asymptotic periodic behaviour \cite{FZZ19}.
Note however this does not imply that the transition time between
the wells is periodic.

We consider specifically the stochastic overdamped Duffing Oscillator
(\ref{eq:DuffingOsc}) as our model of stochastic resonance, this
is a typical model in literature \cite{BenziStochRes,BenziStochRes81,BenziStochRes83,CherubiniStochRes,GHJM98,HerrmannImkeller,HerrmannImkellerPavly}.
It is easy to see that (\ref{eq:DuffingOsc}) is a gradient SDE 
\[
dX_{t}=-\partial_{x}V(t,X_{t})dt+\sigma dW_{t},
\]
derived from the time-periodic double-well potential $V\in C^{1,2}(\mathbb{R\times R})$
given by 
\[
V(t,x)=-\frac{1}{2}x^{2}+\frac{1}{4}x^{4}-Ax\cos(\omega t).
\]
In the absence of the periodic forcing ($A=0$), $V$ has two local
minimas at $x=\pm1$ which are the metastable states and has a local
maxima at $x=0$, the unstable state. We consider the left and right
well to be the intervals $(-\infty,0)$ and $(0,\infty)$ respectively.
Although the local minimas change over time, by the nature of the
problem, we shall normalise the problem to have $x=-1,+1$ as the
bottom of the left and right well respectively. It is well known (see
e.g. \cite{Gardiner_StochMethods,JungPerioidcallyDrivenSys,RiskenFP,Zwanzig})
that in these double-well systems, the process quickly goes to the
bottom of the well (if not already) and stays there for a long period
of time before transitioning to the other well. This remains true
even for $A\neq0$, see \cite{BenziStochRes,BenziStochRes81,BenziStochRes83,NicolisPeriodicForcing,JungPerioidcallyDrivenSys,JungHanggi,ZhouMossJung}.

Currently, there appears to be neither standard nor rigorous definition
of stochastic resonance \cite{HerrmannImkeller,JungHanggi}. In the
context of this paper, a working definition is that the stochastic
system is in stochastic resonance if the noise intensity is tuned
optimally such that the expected transition time between the metastable
states is (approximately) half the period \cite{CherubiniStochRes}.
We approach the stochastic resonance problem by solving the PDE for
many fixed noise intensity values.

For the stochastic resonance problem, we take the starting time $s=0$
(modulo the period), the time in which the right well is at its lowest
point. In other words, we consider the process transition from the
bottom of the right well to the left well. Formally, for $D=(-1,\infty)$
, consider 
\[
\tau_{\sigma}(x)=\inf_{t\geq0}\{X_{t}=-1|X_{0}=x\},\quad x\in D,
\]
associated to the SDE (\ref{eq:DuffingOsc}). We choose the same parameters
$A=0.12$, $\omega=10^{-3}$ and keep the explicit $\sigma$ subscript
as we shall fine-tune $\sigma$ to attain stochastic resonance. The
same parameters was considered in \cite{CherubiniStochRes}.

We first study the sharpness of transition of the particle crossing
the maxima point to reach the potential's minimum. We numerically
solve the PDE (\ref{eq:TimeReversedHomo_IBVP}) for different domains
$D_{b}=(b,3)$ for left boundary point $b=-1.0,-0.9,\cdot\cdot\cdot,0.4,0.5$.
This is shown in Figure \ref{Sharp transition}. The left graph shows
the expected transition times with respect to the different domains
specified for all initial position within that domain. Observe that
as $b$ goes towards $-1$ from $0.5$, the transition times accumulates
quickly. The right graph fixes the initial position at the larger
minima $x=1$ of the double well potential and the curve is the expected
transition time as a function of the left boundary point $b$ of the
domain $D_{b}$. It is the cross section of the left graph at $x=1$.
We see sharp increase of the expected transition time when $b$ approaches
$0$, the maxima of the double well potential, from the right. This
says that the particle would stay on the right hand side of well for
long time. Note however that when $b$ crosses $0$ and towards $-1$,
the expected transition time remains almost unchanged as the domain
expands to $(-1,3)$. This means that once the particle crosses the
maxima, it quickly reaches and settles around the local minima $x=-1$.
This is a reflection of the accumulating curves of the left graph.
Note that the system does not need to exhibit stochastic resonance
for these abrupt transitions and periods of long stability properties
to occur. Furthermore, it is worth noting that this behaviour is expected
by linearising the SDE around its the stationary points. In fact,
we anticipate multiplicative ergodic theorem of non-autonomous random
dynamical systems to rigorously explain this behaviour. Finally, we
note that this this abrupt transition behaviour has been observed
in the cyclical ice age climate change phenomena \cite{BenziStochRes,NicolisPeriodicForcing}.

\begin{figure}[h]
\centering{}\includegraphics[scale=0.5]{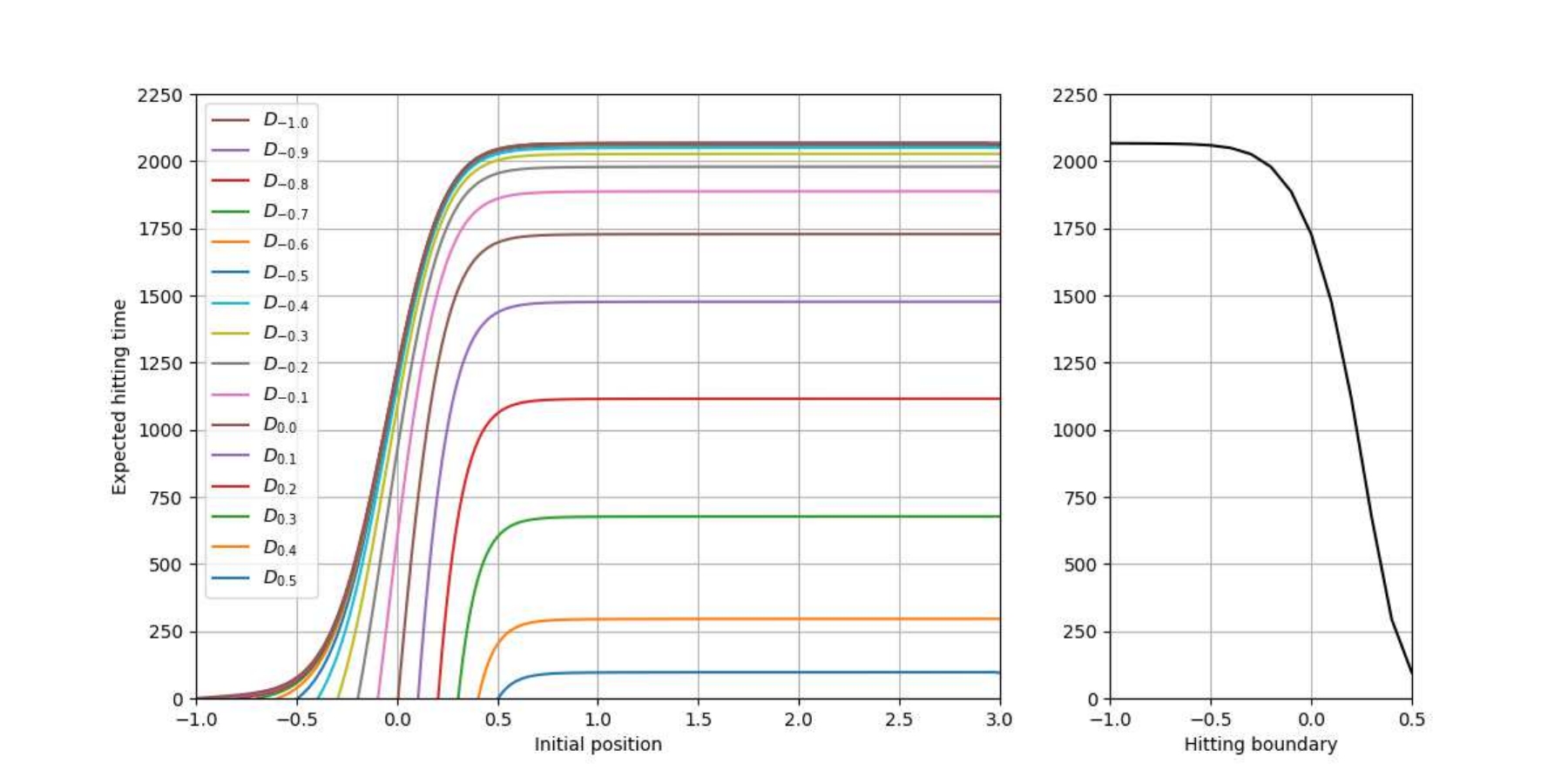} \caption{For SDE (\ref{eq:DuffingOsc}) with parameters $A=0.12,$ $\omega=0.001$,
$\sigma=0.285$, we solve and plot the associated PDE (\ref{eq:TimeReversedHomo_IBVP})
for expected transition time $\bar{\tau}_{0.285}^{\text{bfp}}(x)$
(left) and $\bar{\tau}_{0.285}^{\text{bfp}}(1)$ (right) for domains
$D_{b}=(b,3)$ for $b\in[-1,0.5]$.}
\label{Sharp transition}
\end{figure}

For the stochastic resonance problem, we consider also the transition
from the left well to the right well. Specifically, consider the SDE
\[
\begin{cases}
dY_{t}=\left[Y_{t}-Y_{t}^{3}+A\cos(\omega t)\right]dt+\sigma dW_{t}, & t\geq\frac{T}{2},\\
Y_{\frac{T}{2}}=y,
\end{cases}
\]
and define 
\[
\tau_{\sigma}^{L\rightarrow R}(y)=\inf_{t\geq\frac{T}{2}}\{Y_{t}\in\partial D|Y_{\frac{T}{2}}=y\}-\frac{T}{2},\quad y\in D_{L},
\]
where $D_{L}=(-\infty,1$) and noting that $s=\frac{T}{2}$. Clearly,
by a change of variables, $\tilde{Y}_{t}=-Y_{t+\frac{T}{2}}$ and
since $\cos(\omega(t+\frac{T}{2}))=-\cos(\omega t)$, we have that
\[
\begin{cases}
d\tilde{Y}_{t}=\left[\tilde{Y}_{t}-\tilde{Y}_{t}^{3}+A\cos(\omega t)\right]dt-\sigma d\tilde{W}_{t},\\
\tilde{Y}_{0}=-y,
\end{cases}
\]
where $\tilde{W}_{t}=W_{t+\frac{T}{2}}-W_{\frac{T}{2}}$. It follows
then that $\tau_{\sigma}^{L\rightarrow R}(y)=\inf_{t\geq0}\{\tilde{Y}_{t}\in\partial D|\tilde{Y}_{0}=-y\}$.
Note that $\tilde{W}$ and $W$ have the same distribution, hence
it follows that
\begin{equation}
\bar{\tau}_{\sigma}(x)=\bar{\tau}_{\sigma}^{L\rightarrow R}(-x),\quad x\in D.\label{eq:ExpectedDurationSymmetryOddDrift}
\end{equation}
Indeed the same computation holds provided the drift is an odd function
when $A=0$.

Specifically for SDE\textbf{ }(\ref{eq:DuffingOsc}) where $\omega=0.001$,
$T=2000\pi$ is the period. Given (\ref{eq:ExpectedDurationSymmetryOddDrift}),
it is sufficient to cast the stochastic resonance problem as finding
$\sigma_{*}\neq0$ such that
\begin{equation}
\bar{\tau}_{\sigma_{*}}(1)\simeq\frac{T}{2}=1000\pi.\label{eq:CastingStochRes}
\end{equation}
i.e. the expected transition time between the wells is half the period.

To fine tune for stochastic resonance, we repeatedly solve the same
PDE computations with the same numerical parameters and methods (as
for Figure \ref{SDE_vs_PDE_plot}), changing only $\sigma$ and considering
the expected transition time $\bar{\tau}_{\sigma}^{\text{\text{\text{grad}}}}(1)$
as a function of $\sigma$. We vary $\sigma$ in the $\sigma$-domain
$[0.2,1]$. We partition this $\sigma$-domain into two subintervals
$[0.2,0.3]$ and $[0.3,1]$ and uniformly partition them into 50 and
100 points respectively. As a function of $\sigma$, we plot the expected
transition time $\bar{\tau}_{\sigma}(1)$ in Figure \ref{StochResFineTune_plot}.

\begin{figure}[h]
\centering{}\includegraphics[scale=0.5]{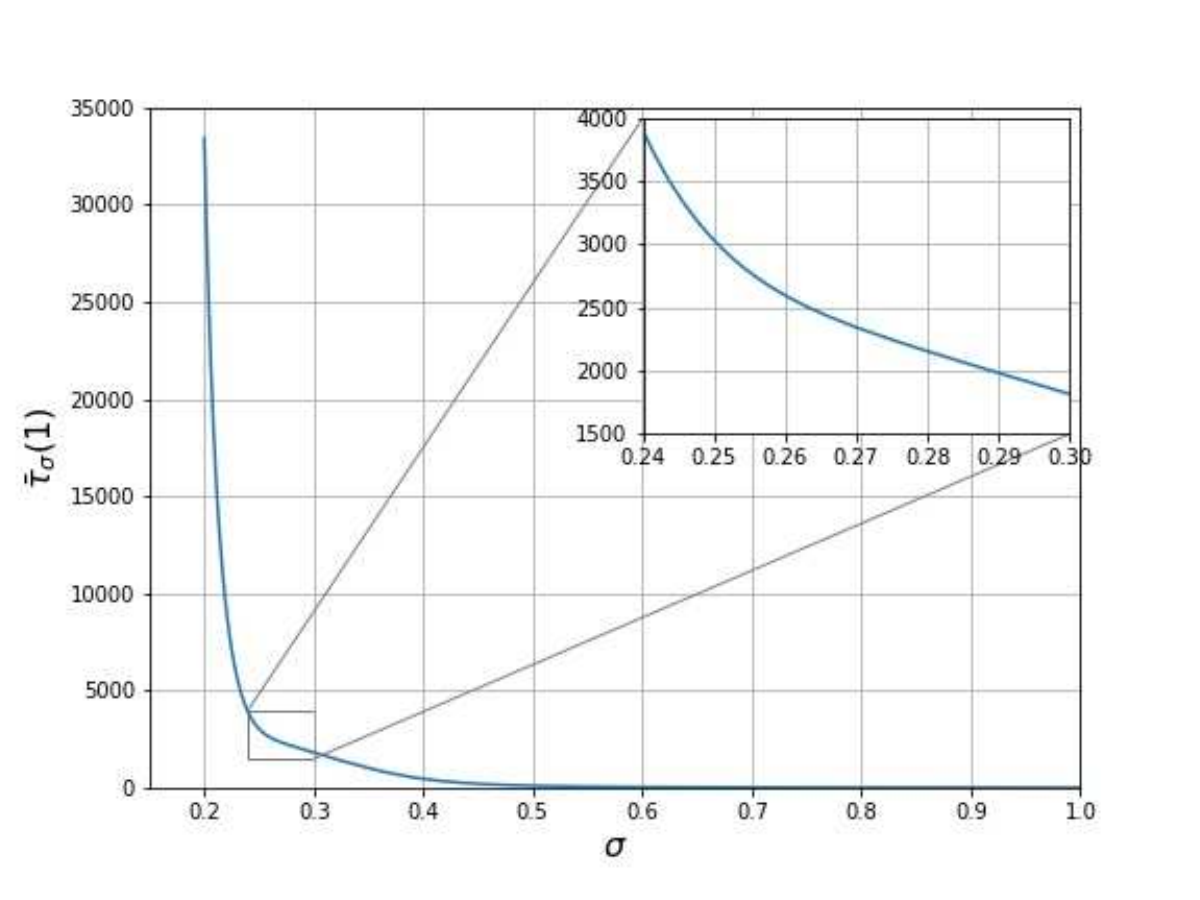}\caption{Plot of $\bar{\tau}_{\sigma}^{\text{\text{\text{grad}}}}(1)$ for
$\sigma\in[0.2,1]$ for SDE (\ref{eq:DuffingOsc}) with parameters
$A=0.12$ and $\omega=0.001$, $s=0$ and $\tilde{D}=(-1,3)$.}
\label{StochResFineTune_plot} 
\end{figure}

It can be seen from Figure \ref{StochResFineTune_plot} that (\ref{eq:CastingStochRes})
is satisfied for some $\sigma_{*}\in[0.245,0.25]$. We compute further
$\bar{\tau}_{\sigma}^{\text{\text{\text{grad}}}}(1)$ on a finer partition
of the interval $[0.245,0.25]$ further and tabulate its numerical
values in Table \ref{StochResFineTune_table}. Numerically, from Table
\ref{StochResFineTune_table}, it can be seen that $\bar{\tau}_{0.2485}^{\text{grad}}\simeq\frac{T}{2}=1000\pi$
to the nearest $5\cdot10^{-4}$.
\noindent \begin{flushleft}
\begin{table}
\begin{centering}
\begin{tabular}{|c|c|c|c|c|c|c|c|c|c|c|c|}
\hline 
$\sigma$  & 0.245  & 0.2455  & 0.246  & 0.2465  & 0.247  & 0.2475  & 0.248  & 0.2485  & 0.249  & 0.2495  & 0.25\tabularnewline
\hline 
$\bar{\tau}_{\sigma}^{\text{\text{\text{grad}}}}(1)$  & 3388  & 3346  & 3306  & 3267  & 3230  & 3194  & 3159  & \multicolumn{1}{c|}{3125} & 3093  & 3061  & 3030\tabularnewline
\hline 
\end{tabular}
\par\end{centering}
\caption{Computation of $\bar{\tau}_{\sigma}^{\text{\text{\text{grad}}}}(1)$
(4sf) on a finer partition of $\sigma\in[0.245,0.25].$}
\label{StochResFineTune_table}
\end{table}
\par\end{flushleft}

\section{Discussion}

In this paper, we rigorously derived parabolic PDE (\ref{eq:ExpectedDurationPDE_Periodic})
that the expected duration of time-periodic non-degenerate SDEs satisfies,
complete with time-periodic boundary conditions and assumptions on
the SDE coefficients. Furthermore, we proved that the PDE is well-posed
by casting the problem as a fixed point problem in Theorem \ref{thm:ExistenceUniquenessResult_Coercive}
as well as a convex optimisation problem in Theorem \ref{thm:ConvexThm}.
In Section \ref{sec:NumericalSolutionOfPde}, we provided how these
approaches can be applied numerically to solve PDE (\ref{eq:ExpectedDurationPDE_Periodic}).
In Section \ref{subsec:Stochastic-Resonance}, we apply our theory
and provide numerics to compute the expected duration time to switch
from the regime and how it varies with the change of the noise intensity.
This provides a fine-tuning method to find the noise intensity from
historical events in reality. While it works very well for stochastic
resonance problem, our theory is applicable to many real world and
scientific problems where random periodicity is ubiquitous and exit
duration is of important consideration.

Some conditions was assumed in this paper to simplify the exposition
and to focus on the main ingredients and techniques used to attain
the results given in this paper. For example, it was assumed the diffusion
coefficient of SDE (\ref{eq:GeneralSDE_NonAuton}) is non-degenerate.
We believe that this assumption can be relaxed along the lines of
(time-dependent) Hörmander's condition and possibly even the UFG (Uniformly
Finitely Generated) condition \cite{KS82}. Similarly, we anticipate
that the results in the paper can be extended to Lévy noise with a
different infinitesimal generator. Another example is that in Theorem
\ref{thm:ConvexThm} we posed the solution of the (\ref{eq:ExpectedDurationPDE_Periodic})
as convex optimisation problem in the $L^{2}(D)$ Hilbert space. We
note that Lemma \ref{lem:ConvexLemma} applies more generally to reflective
Banach spaces such as $L^{p}(D)$ for $1<p<\infty$. If Theorem \ref{thm:ConvexThm}
can be established for $L^{p}(D)$, by Sobolev embedding, this would
imply that (\ref{eq:ExpectedDurationPDE_Periodic}) is well-posed
for a wide range of systems and in higher dimensional systems.

On a broader view, there are a few different directions that this
paper opens up. For instance, in Section \ref{subsec:Stochastic-Resonance},
we implicitly assumed that the mapping $\sigma\rightarrow\bar{\tau}_{\sigma}(s,x)$
is continuous in the fine-tuning of $\sigma$ for the stochastic resonance
problem. While we expect this to be true, a rigorous proof would be
of interest. From a theory perspective, we recall that Theorem \ref{thm:ExpectedDurationPDE_Theorem}
can be regarded as an instance of time-periodic Feynman-Kac duality
in relating the expected duration of time-periodic SDEs to time-periodic
solution of a parabolic PDE. We anticipate using a similar approach
of this paper, other time-periodic Feynman-Kac dualities exist for
other quantities of time-periodic SDEs.

\bigskip{}

\textbf{Acknowledgements. }We would like to thank the referees for
their constructive comments which result in improvement of this paper.
We are grateful to David Sibley for his help with numerical schemes.
We would like to acknowledge the financial support of a Royal Society
Newton fund grant (ref. NA150344) and an EPSRC Established Career
Fellowship to HZ (ref. EP/S005293/1)

\end{document}